\documentclass[12pt]{amsart}
\usepackage{geometry}
\usepackage{graphicx}
\usepackage{amssymb}
\usepackage{epstopdf}
\usepackage{amsmath,amscd}
\usepackage{amsthm}
\usepackage{url,verbatim}
\usepackage{mathtools}

\usepackage{etoolbox}

\RequirePackage[colorlinks,citecolor=blue,urlcolor=blue]{hyperref}
\usepackage{breakurl}
\theoremstyle{plain}
\newtheorem{theorem}{Theorem}
\newtheorem{definition}[theorem]{Definition}
\newtheorem{lemma}[theorem]{Lemma}
\newtheorem{proposition}[theorem]{Proposition}

\newtheorem{example}[theorem]{Example}
\newtheorem{remark}[theorem]{Remark}

\newcounter{mycount}

\newenvironment{numlist}{\begin{list}{\arabic{mycount}.}%
   {\usecounter{mycount}\labelwidth=1cm\itemsep 0pt}}{\end{list}}
\newenvironment{letlist}{\begin{list}{\rm(\alph{mycount})}%
   {\usecounter{mycount}\labelwidth=1cm\itemsep 0pt}}{\end{list}}

\numberwithin{equation}{section}
\numberwithin{theorem}{section}
\numberwithin{figure}{section}

\newcommand\HH{{\mathbb H}}
\newcommand\RR{{\mathbb R}}
\newcommand\PP{{\mathbb P}}

\newcommand\qq{\qquad}
\newcommand\q{\quad}
\newcommand\si{\sigma}
\newcommand\be{\beta}
\newcommand\al{\alpha}
\newcommand\Si{\Sigma}
\newcommand\g{\gamma}

\newcommand\Om{\Omega}
\newcommand\De{\Delta}
\newcommand\ol{\overline}
\newcommand\FF{{\mathbb F}}

\newcommand\sG{{\mathcal G}}

\newcommand\CC{{\mathbb C}}
\newcommand\sB{{\mathcal B}}
\newcommand\NN{{\mathbb N}}
\newcommand\ZZ{{\mathbb Z}}

\newcommand\VV{{\mathbb V}}
\newcommand\EE{{\mathbb E}}

\newcommand\sF{{\mathcal F}}

\newcommand\La{\Lambda}
\newcommand\la{\lambda}
\newcommand\eps{\epsilon}
\newcommand\ot{1-2\ }
\newcommand\resp{respectively}
\newcommand\deqd{\doteqdot}
\newcommand\lra{\leftrightarrow}
\newcommand\oo{\infty}
\newcommand\wt{\widetilde}
\newcommand\TT{{\mathbb T}}
\renewcommand\th{\theta}
\newcommand\de{\delta}

\newcommand\Pf{\mathrm{Pf}\, }

\newcommand\rv{\mathrm{v}}
\newcommand\re{\mathrm{e}}
\newcommand\bec{\be_{\mathrm{c}}}
\newcommand\es{\varnothing}

\newcommand\m{\mathrm{mix}}
\newcommand\e{\mathrm{ev}}
\newcommand\I{\mathrm{I}}
\newcommand\HnD{\HH_{n,\De}}

\newcommand\wh{\widehat}
\newcommand\sgn{\mathrm{sgn}}
\newcommand\mo{\mu_\oo}
\renewcommand\ell{l}
\newcommand\tr{\mathrm{Tr}}
\newcommand\Log{\mathrm{Log}}
\renewcommand\mod{{\ \textrm{mod}\ }}
\newcommand\ZNI{Z_n(I)}

\newcommand\sA{{\mathcal A}}

\title{Critical surface of the \ot model}
\author{Geoffrey R.\ Grimmett}
\address{Statistical Laboratory, Centre for
Mathematical Sciences, Cambridge University, Wilberforce Road,
Cambridge CB3 0WB, UK} 

\email{g.r.grimmett@statslab.cam.ac.uk}
\urladdr{\url{http://www.statslab.cam.ac.uk/~grg/}}

\author{Zhongyang Li}
\address{Department of Mathematics,
University of Connecticut,
Storrs, Connecticut 06269-3009, USA} 
\email{zhongyang.li@uconn.edu}
\urladdr{\url{http://www.math.uconn.edu/~zhongyang/}}

\begin{document}

\begin{abstract}
The \ot model on the hexagonal lattice is a model of statistical mechanics in which
each vertex is constrained to have degree either $1$ or $2$. There are three 
edge-directions, and three corresponding parameters $a$, $b$, $c$. It is
proved that, when $a \ge b \ge c > 0$, the  surface
given by $\sqrt a = \sqrt b + \sqrt c$ is critical. The proof hinges upon 
a representation of the partition function in terms of that of
a certain dimer model.  This dimer model may be studied
via the Pfaffian representation of Fisher, Kasteleyn, and Temperley. 
It is proved, in addition, that the two-edge correlation function converges exponentially 
fast with distance
when $\sqrt a \ne \sqrt b + \sqrt c$.
Many of the results may be extended to periodic models.
\end{abstract}

\date{28 June 2015, revised 1 June 2016, 17 February 2017}

\maketitle

\section{Introduction and background}\label{sec:intro}

The \ot model on the hexagonal lattice was introduced by Schwartz and Bruck \cite{SB08}
as an intermediary in the calculation of the capacity of a constrained coding system.
They expressed the capacity via holographic reductions (see \cite{Val}) in terms
of the number of perfect matchings (or dimer configurations), and the latter may be
studied via the Pfaffian method of Fisher, Kasteleyn, and Temperley
\cite{F61,Kast61,TF61}. The \ot model may  be viewed as a model of statistical
mechanics of independent interest, and it is related to the Ising model and the dimer model. 
In the current paper, we study the
\ot model within this context, and we establish the exact form of
the associated critical curve.

A \ot configuration on the hexagonal lattice $\HH=(\VV,\EE)$ 
is a subset $F$ of edges such that every
vertex is incident with either one or two edges of $F$. There are three real parameters $a,b,c>0$,
which are associated with the three classes of edges of $\HH$.
The weight of a configuration on a finite region is the product over vertices $v$ of
one of $a,b,c$ chosen according to the edge-configuration at $v$. (See Figure
\ref{fig:sign}.) 

Through a sequence of transformations, the \ot model turns out to be linked to
an enhanced Ising model, a polygon model, and a dimer model. These connections are pursued
here, and in the linked paper \cite{GL7}. 
The main result (Theorem \ref{thm:main}) states in effect that, when $a\ge b, c>0$, the
surface given by $\sqrt a = \sqrt b + \sqrt c$ is critical. This is proved
by an analysis of the behaviour of the two-edge correlation function
$\langle \si_e\si_f\rangle$ as $|e-f|\to\oo$. The model is called \emph{uniform} if
$a=b=c=1$, and thus the uniform model is not critical in the above sense.

There has been major progress in recent years in the study of two-dimensional Ising
models via rhombic tilings and discrete holomorphic observables (see, for example, 
\cite{BdT12, ChelkS2, ChelkS, Ken04}).
There is a rhombic representation of the critical polygon model 
associated with the \ot model, and an associated
discrete holomorphic function, but this is not explored here. 

Certain properties of the underlying hexagonal lattice are utilized heavily in this work, such as
trivalence, planarity, and support of a $\ZZ^2$ action. It may be possible
to extend many of the results of this paper to certain other graphs with such properties, including
the Archimedean lattice  $(3, 12^2)$ and the square/octagon lattice $(4, 8^2)$.
Further extensions are possible to periodic models on hexagonal and other lattices. 
(See Remarks \ref{rem:-1}, \ref{rem:alt} and Section \ref{sec:lis2}.) 

It was shown already in \cite{ZL2} that a (geometric) phase transition exists for the \ot model on $\HH$. 
An \emph{$a$-cluster} is a connected set of vertices each having local weight $a$ (as above).
It was shown that there exists, a.s.\ with respect to any 
translation-invariant Gibbs measure,
no infinite path  of present edges. In contrast, for given $b$, $c$, there exists
no infinite $a$-cluster for small $a$, whereas such a cluster exists for large $a$.
The a.s.\ uniqueness of infinite `homogeneous' clusters was proved in \cite{ZL3}.

This paper is concentrated on the \ot model and its dimer representation.
A related representation involves the  polygon model on $\HH$, and 
the phase transition of the latter model is the subject of the 
linked paper \cite{GL7}. The polygon
representation is related to the high temperature
expansion of the Ising model, and results in an inhomogeneous
model that may  regarded as an extension of the $O(n)$ model with $n=1$;
see \cite{DPSS} for a recent reference to the $O(n)$ model.

The structure of the current work is as follows. The precise formulation
of the \ot model appears in Section \ref{sec:model}, and the main theorem
(Theorem \ref{thm:main}) is presented in Section \ref{sec:mainthm}.

The \ot model is coupled with an Ising model in Section \ref{partf}, in a manner
not dissimilar to the Edwards--Sokal coupling of the random-cluster model
(see \cite[Sect.\ 1.4]{G-RCM}). 
It may be transformed into a dimer model (see \cite{ZL2}) as described in
Section \ref{sec:dimer}.
In Section \ref{sec:free}, we gather some conclusions about infinite-volume 
free energy and infinite-volume measures
that are new for the \ot model. Theorem \ref{thm:main}
is proved in Sections \ref{sec:morepf}--\ref{sec:pf32} by an analysis using Pfaffians,
and further in Sections \ref{sec:pf-1} and \ref{sec:eecad}.
Section \ref{sec:periodic} is devoted to extensions of
the above results to periodic 1-2 and Ising models to which the
Kac--Ward approach of \cite{Lis} does not appear to apply.

\section{The \ot model}\label{sec:model}

Let $G=(V,E)$ be a finite graph. A \emph{\ot configuration} on $G$ is
a subset $F\subseteq E$ such that every $v \in V$ is incident to either one or two members
of $F$. The subset $F$ may be expressed as a vector in the space $\Si=\{-1,+1\}^E$
where $-1$ represents an absent edge and $+1$ a present edge. Thus the
space of \ot configurations may be viewed as the subset of $\Si$ containing
all vectors $\si$ such that
$$
\sum_{e\ni v} \si'_e \in \{1,2\}, \qq v \in V,
$$
where
\begin{equation}\label{eq:sigmap}
\si'(e) = \tfrac12(1+\si(e)).
\end{equation}
(In Section \ref{ssec1}, we will write $\Si^\re$ for $\Si$, in order
to distinguish it from a space of vertex-spins to be denoted $\Si^\rv$.)

\begin{figure}[htbp]
\centerline{\includegraphics*[width=0.45\hsize]{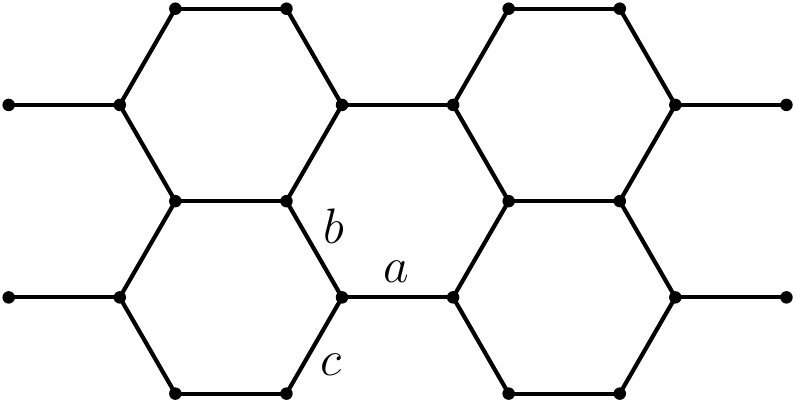}}
   \caption{An embedding of the hexagonal lattice. Horizontal edges
   are said to be of type $a$, NW edges of type $b$, and NE edges of type $c$.}
   \label{fig:hex}
\end{figure}

Suppose now that $G$ is a finite part of the hexagonal lattice $\HH$, suitably embedded in $\RR^2$,
see Figure \ref{fig:hex}.
The embedding is such that each edge may be viewed as one of: horizontal, NW, or NE. (Later we
shall consider a finite box with toroidal boundary conditions.) 
Let $a,b,c \ge 0$ be such that $(a,b,c)\ne(0,0,0)$, 
and associate these three parameters with the edges 
as indicated in the figure.  For $\si\in\Si$ and $v \in V$, let $\si|_v$ be
the sub-configuration of $\si$ on the three edges incident to $v$. 
There are $2^3=8$ possible local configurations,
which we encode as words of length three in the alphabet with letters $\{0,1\}$.
That is, for $v \in V$,
we observe the states $\si(e_{v,a}),\si(e_{v,b}), \si(e_{v,c})$,
where  
$e_{v,a}$, $e_{v,b}$, $e_{v,c}$ are the edges of type $a$, $b$, $c$ (\resp) incident to $v$.
The corresponding \emph{signature} $s_v$ is the word $\si'(e_{v,c})\si'(e_{v,b})\si'(e_{v,a})$ of length $3$,
where $\si'$ is given in \eqref{eq:sigmap}.
That is, the signature of $v$ is given as in Figure \ref{fig:sign}, together with the local weight
$w(\si|_v)$ associated with each of the eight possible signatures.

The hexagonal lattice $\HH$ is, of course, bipartite, and we colour the two
vertex-classes \emph{black} and \emph{white}. The upper diagrams of Figure
\ref{fig:sign} are for black vertices, and the lower for white vertices.

\begin{figure}[htbp]
\centerline{\includegraphics*[width=0.98\hsize]{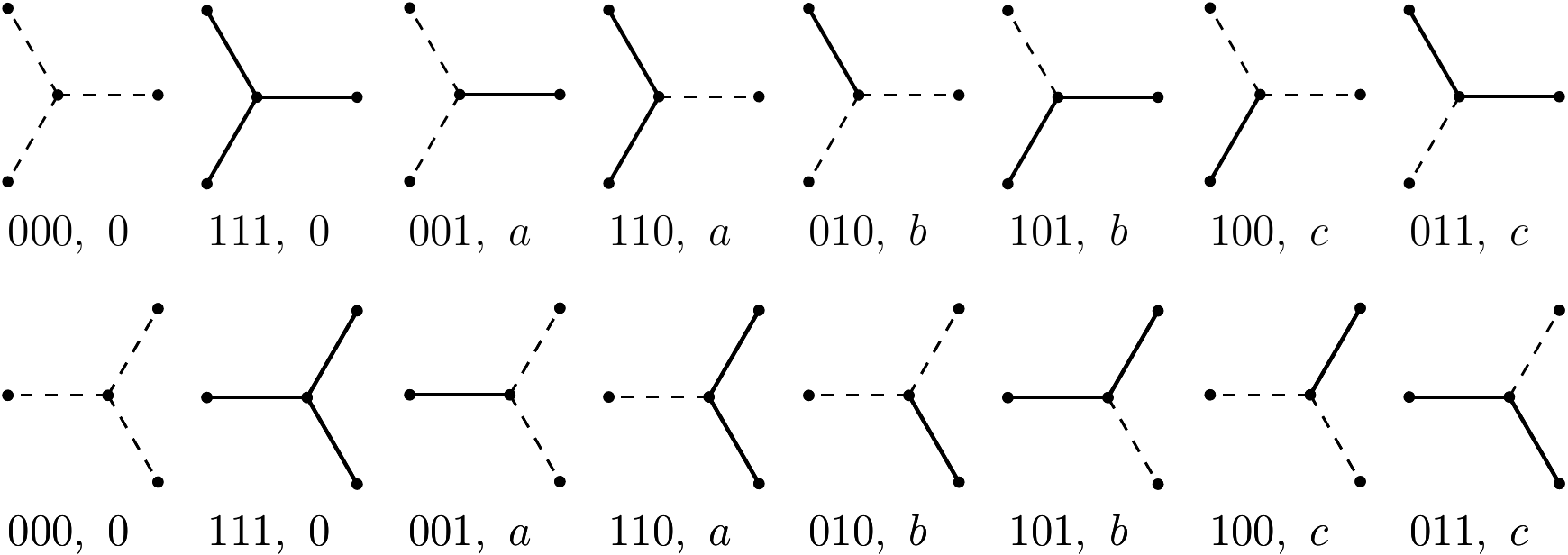}}
   \caption{The eight possible local configurations $\si|_v$ at a vertex $v$  in the two
   cases of \emph{black} and \emph{white} vertices. The signature
   of each is given, and also the local weight $w(\si|_v)$ associated with each instance.}
   \label{fig:sign}
\end{figure}

To the vector $\si\in\Si$, we assign the weight 
\begin{equation}\label{eq:pf-2}
w(\si) = \prod_{v\in V} w(\si|_v).
\end{equation}
These weights give rise to the partition
function
\begin{equation}\label{eq:pf-1}
Z=\sum_{\si\in\Si} w(\si),
\end{equation}
which leads in turn to the probability measure  
\begin{equation}\label{eq:pm}
\mu(\si) = \frac1Z  w(\si), \qq\si\in\Si.
\end{equation}
It is easily seen that the measure $\mu$ is invariant under
the mapping $(a,b,c)\mapsto (ka,kb,kc)$ with $k >0$.
It is therefore natural to re-parametrize the \ot model by
\begin{equation}\label{eq:ratio}
(a',b',c')=\frac{(a,b,c)}{\|(a,b,c)\|_2}.
\end{equation}

We will work mostly with a finite subgraph of $\HH$ subject to toroidal boundary conditions.
Let $n \ge 1$, and let $\tau_1$, $\tau_2$ be the two shifts of $\HH$, illustrated in Figure \ref{fig:hex0},
that map an elementary hexagon to the next hexagon in the given directions.
The pair $(\tau_1,\tau_2)$ generates a $\ZZ^2$ action on $\HH$, and we write $\HH_n$ for
the quotient graph of $\HH$ under the subgroup of $\ZZ^2$ generated 
by $\tau_1^n$ and $\tau_2^n$.  The resulting $\HH_n$ is illustrated in Figure
\ref{fig:hex0}, and may be viewed as a finite subgraph of $\HH$ subject to toroidal
boundary conditions.

\begin{figure}[htbp]
\centering
\scalebox{1}[1]{\includegraphics*[width=0.6\hsize]{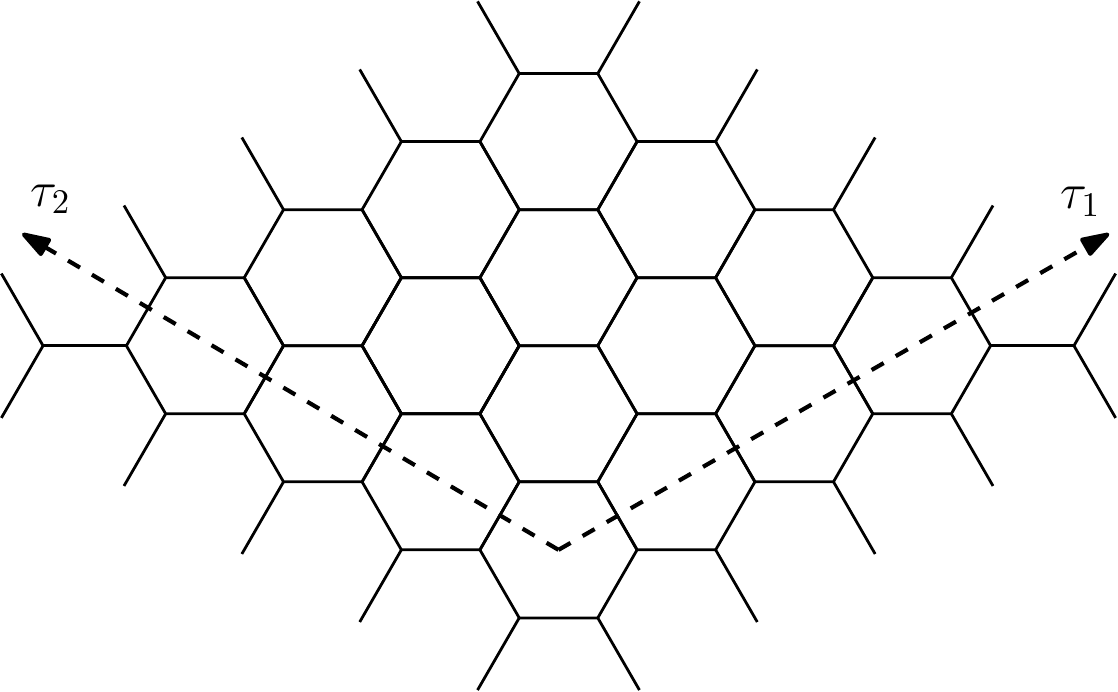}}
\caption{The graph $\HH_n$ is an $n\times n$ `diamond' wrapped onto a torus,
as illustrated here with $n=4$.}\label{fig:hex0}
\end{figure}

Our purpose in this paper is to study the \ot measure  \eqref{eq:pm} on $\HH_n$
in the infinite-volume limit as $n\to\oo$, and to identify its critical surface.
As an indicator of phase transition, we shall use the two-point function
$\langle \si_e\si_f\rangle_n$, where $e$, 
$f$ are two edges and $\langle \cdot\rangle_n$ denotes expectation.

We do not explore in detail the nature 
and multiplicity of infinite-volume measures in this paper.
There are certain complexities in such issues arising from the absence of
a correlation inequality, and some partial results along these lines may be found in \cite[Thm 0.1]{ZL2}.
These results are developed in Section \ref{sec:free}, where
the main result of current value is the existence of the 
infinite-volume limit of the toroidal \ot measure, see Theorem \ref{es}. 

\section{Main results}\label{sec:mainthm}

Consider the \ot model on 
$\HH_n$ with parameters $a,b,c >0$. 
We write $e=\langle x,y\rangle$ for the edge $e$ with endpoints $x$, $y$, 
and we use $\langle X\rangle_n$ to denote expectation of the random variable $X$ 
with respect to the probability measure of \eqref{eq:pm} on $\HH_n$. 
We shall make use of a measure of distance $|e-f|$ between $e$ and $f$, and it is largely immaterial
which measure we take. For definiteness, consider $\HH$ embedded in $\RR^2$ in the manner of Figure
\ref{fig:hex0}, with unit edge-lengths, and let $|e-f|$ be the Euclidean distance between their midpoints. 

We shall sometimes require the following 
geometric condition on two NW edges $e,f \in \EE$: 
\begin{equation}\label{eq:condition0}
\begin{aligned}
&\text{there exists a path $\pi=\pi(e,f)$ of $\HH_n$ from $e$ to $f$}\\
& \text{using only horizontal and NW half-edges}.
\end{aligned}
\end{equation}

\begin{theorem}\label{thm:main}
Let $a, b, c >0$, and $e,f\in\EE$.
\begin{letlist}
\item 
The limit $\langle \si_e\si_f\rangle=\lim_{n\to\oo}\langle \si_e\si_f\rangle_n$ exists.
\item \emph{Subcritical case.}
Let $a \ge b >0$ and $\sqrt a - \sqrt b < \sqrt c< \sqrt a + \sqrt b$. There exists
$\al(a,b,c)>0$ such that
\begin{equation}\label{eq:expfast}
|\langle \si_e\si_f\rangle| \le e^{-\al|e-f|}, \qq e,f\in\EE.
\end{equation}
\item \emph{Supercritical case.}
Let $a \ge b>0$, and let $e$, $f$ be NW edges satisfying \eqref{eq:condition0}. 
For almost every $c>0$ satisfying either $\sqrt a > \sqrt b + \sqrt c$ or $\sqrt{c}>\sqrt{a}+\sqrt{b}$,  we have that
$\lim_{|e-f|\to\oo}\langle \si_e\si_f\rangle^2$ exists and is non-zero.
The convergence is exponentially fast in the distance $|e-f|$.
\end{letlist}
\end{theorem}

The two-edge function $\langle \si_e\si_f\rangle$ behaves 
(when $a\ge b > 0$) in a qualitatively different
manner depending on whether or not $\sqrt a-\sqrt b < \sqrt c<\sqrt a+\sqrt b$. 
Here is a motivation for condition \eqref{eq:condition0}.
Consider the `ground states' when either $c=0$ or $a=b=0$. By examination  of
the different cases in Figure \ref{fig:sign}, we may see,
subject to \eqref{eq:condition0}, that 
\begin{equation}\label{eq:extreme}
\langle \si_e\si_f\rangle = 1\qq\text{if either}\q a,b>0,\ c=0,\q \text{or}\q a=b= 0,\ c>0.
\end{equation}
The result of part (c) will follow from this by an argument using analyticity
(and, moreover, the set of $c$ at which the conclusion of (c) fails is
a union of isolated points). 
Part (c) holds with $e$, $f$ assumed to be horizontal rather than NW.

Theorem \ref{thm:main} is not of itself 
a complete picture of the location of critical phenomena
of the \ot model, since the conditions on the parameters in part (c) are allied to the
direction of the vector from $e$ to $f$. 
(The direction NW is privileged in the above theorem. Similar results hold for
the other two lattice directions with suitable permutations of the parameters.) 
We have not ruled out the theoretical possibility of 
further critical surfaces in the parameter-space $[0,\oo)^3$.

\begin{remark}\label{rem:-1}
The quickest proof of Theorem \ref{thm:main}(b), the subcritical case,
 (given in Section \ref{sec:pf-1}) 
is based on a result of \cite{Lis} that imposes a condition on the
parameters of edges incident to a vertex $v$, uniformly in $v$. This
condition is
satisfied in the current setting (see Section \ref{sec:pf-1}).  In
the more general setting of certain periodic but non-constant families of parameters,
or possibly of the 1-2 model on other graphs such as the square/octagon lattice, 
much of Theorem \ref{thm:main} remains true, but the condition of \cite{Lis} does
not generally hold.  In  order to overcome this lacuna for more general systems,
we present a further proof of Theorem \ref{thm:main}(b) in Section \ref{sec:eecad}
(in the more general form of Theorem \ref{thm:main4})
using the dimer-related techniques of the proof of Theorem \ref{thm:main}.  
Such results may be extended in part to more general periodic settings, see
Section \ref{sec:lis2}.
\end{remark}

The proof of Theorem \ref{thm:main} utilizes 
a sequence of transformations between the \ot model
and the Ising and dimer models, as described in the
forthcoming sections. Theorem \ref{thm:main}(b) is proved in Section \ref{sec:pf-1}.
Most of the remaining proof is found 
in Section \ref{sec:morepf}, with the exponential rate of part (c)
proved in Section \ref{sec:pf32}. The last is proved via
a general result concerning the convergence rate of the determinants
of large truncated block Toeplitz matrices to their limit when the symbol is a smooth matrix-valued
function on the unit circle.

\section{Spin representations of the \ot model}\label{partf}

Two spin representations of the \ot model are presented here.
In the first, the \ot partition function is rewritten in terms
of edge-spins. The second is reminiscent of the random-cluster
representation of the Potts model. A further set of
spin-variables are introduced at the vertices of the graph,
together with an Ising-type partition function.  

\subsection{The \ot model as a spin system}

Let $\HH_n=(V_n,E_n)$ be the quotient hexagonal lattice embedded in the torus
in the manner of Figure \ref{fig:hex0}. 
Let $\Si_n=\{-1,+1\}^{E_n}$, where $-1$ (\resp, $+1$) represents an absent edge  (\resp, present edge). 

For $\si\in\Si_n$ and $v\in V_n$, let $\si_{v,a}$, $\si_{v,b}$, $\si_{v,c}$ denote the spins on the incident $a$-edge, $b$-edge,
$c$-edge of $v$. 
Two partition functions $Z$, $Z'$ 
generate the same measure whenever they differ only in a multiplicative factor 
(that is, their weight functions satisfy $w(\si)=cw'(\si)$
for  some $c \ne 0$ and all $\si\in\Si$), in which case we write $Z \deqd Z'$.  
We represent the \ot model as a spin system as follows.

\begin{proposition}\label{prop:12}
Let $a,b,c\ge 0$ such that $(a,b,c)\ne (0,0,0)$.
The \ot model with parameters $a$, $b$, $c$ on $\HH_n$ has partition function $Z_n$
satisfying $Z_n \deqd Z_n'$ where
\begin{equation}\label{pft}
Z_n':=\sum_{\si\in\Si_n}\,\prod_{v\in V_n}
\bigl(1+A\si_{v,b}\si_{v,c}+B\si_{v,a}\si_{v,c}+C\si_{v,a}\si_{v,b}\bigr),
\end{equation}
and
\begin{align}
A=\frac{a-b-c}{a+b+c},\q 
B=\frac{b-a-c}{a+b+c},\q 
C=\frac{c-a-b}{a+b+c}.\label{abc}
\end{align} 
\end{proposition}

\begin{proof}
By examination of \eqref{pft}, we see that a vertex with local configuration labelled $a$
in Figure \ref{fig:sign} has weight
$$
1+A-B-C = \frac{4a}{a+b+c},
$$
with similar expressions for vertices with the other possible signatures.
This is in agreement with \eqref{eq:pf-2}--\eqref{eq:pf-1}, and the claim follows.
\end{proof}

\subsection{Coupled Ising representation}\label{ssec1}

Let $A\HH_n=(A V_n, A E_n)$ be the graph derived 
from $\HH_n=(V_n,E_n)$ by adding
a vertex at the midpoint of each edge in $E_n$. Let $M E_n=\{M e: e \in E_n\}$ be the set
of such midpoints, and $A V_n = V_n \cup M E_n$.  The edges $A E_n$
are precisely the half-edges of $E_n$, each being of the form $\langle v, Me\rangle$
for some $v \in V_n$ and incident edge $e \in E_n$.  

We introduce an Ising-type model on the graph $A\HH_n$. The marginal
of the model on midpoints $ME_n$ is a \ot model, and the marginal on $V_n$ is an Ising model.
This enhanced Ising model is reminiscent of the coupling
of the Potts and random-cluster measures, see \cite[Sect.\ 1.4]{G-RCM}.
It is constructed initially via a weight function on configuration space,
and via the associated partition function. The weights may be complex-valued, and thus
there  does not always exist an associated probability measure. 

The better to distinguish between
$V_n$ and $M E_n$, we set $\Si_n^\re=\{-1,+1\}^{ME_n}$ as before, and $\Si_n^\rv =\{-1,+1\}^{V_n}$.
An edge $e\in E_n$ is identified with the element of 
$ME_n$ at its centre.
A spin-vector is a pair $(\si^\re,\si^\rv)\in \Si^\re \times \Si^\rv$ with 
$\si^\re=(\si_{v,s}: v\in V_n,\ s=a,b,c)$ and
$\si^\rv=(\si_v: v \in V_n)$,
to which we allocate the (possibly negative, or even complex)  weight 
\begin{equation}\label{eq:ss}
\prod_{v\in V_n}(1+\eps_a\si_v\si_{v,a})(1+\eps_b\si_v\si_{v,b})(1+\eps_c\si_v\si_{v,c}),
\end{equation}
where $\eps_a, \eps_b, \eps_c\in\CC$ are 
constants associated with horizontal, NW, and NE edges, respectively, 
and $\si_{v,a},\si_{v,b},\si_{v,c}$ denote the spins on midpoints of the corresponding 
edges incident to $v\in V_n$. If $u$ and $v$ are endpoints of the same edge 
$\langle u, v \rangle$ of $\HH_n$, then $\si_{u,a}=\si_{v,a}$. 
In \eqref{eq:ss}, each factor $1+\eps_s \si_v\si_{v,s}$ ($s=a,b,c$) 
corresponds to a half-edge of $\HH _n$.
Recalling that
\begin{equation}\label{eq:tanh}
e^{x\si_1\si_2}=(1+\si_1\si_2\tanh x)\cosh x, \qq x \in \RR,\ \si_1\si_2=\pm1,
\end{equation}
the above spin system is  
a ferromagnetic Ising model on $A\HH _n$ when $\eps_a,\eps_b,\eps_c\in(0,1)$.

\subsection{Marginal on the midpoints $ME_n$}\label{ssec2}

The partition function  of \eqref{eq:ss} is
\begin{equation}\label{eq:pfss}
\ZNI:=\sum_{\si^\re\in\Si_n^\re}\, \sum_{\si^\rv \in \Si_n^\rv}\, \prod_{v\in V_n}(1+\eps_a\si_v\si_{v,a})(1+\eps_b\si_v\si_{v,b})(1+\eps_c\si_v\si_{v,c}).
\end{equation}
(The notation $\ZNI$ is chosen for consistency with the polygon partition
function $Z_n(P)$ used in this article and imported from \cite{GL7}.)
The product, when expanded, is a sum of monomials in which each $\si_v$ has a power
between $0$ and $3$. On summing over $\si^\rv$, only terms with even powers
of the site-spins $\si_v$ survive, and furthermore $\si_v^2=1$, so that
\begin{equation*}
\ZNI=2^{|V_n|} \sum_{\si^\re\in\Si_n^\re}\,\prod_{v\in V_n}
\bigl(1+\eps_b\eps_c\si_{v,b}\si_{v,c}+\eps_a\eps_c\si_{v,a}\si_{v,c}+\eps_a\eps_b\si_{v,a}\si_{v,b}
\bigr).
\end{equation*}

Let $a,b,c>0$ be such that $ABC\ne 0$ where $A$, $B$, $C$ are given by \eqref{abc},
and let
\begin{equation}
\eps_a=\sqrt{\frac{BC}{A}},
\q \eps_b=\sqrt{\frac{AC}{B}},
\q\eps_c=\sqrt{\frac{AB}{C}}.
\label{ce}
\end{equation}
By \eqref{pft}, 
\begin{equation}\label{eq:new=}
\ZNI =2^{|V_n|}  Z_n',
\end{equation}
whence the marginal model of \eqref{eq:ss} on the 
midpoints of edges of $\HH _n$, subject to \eqref{ce},  is simply the \ot model
with parameters $a$, $b$, $c$. 

\subsection{Marginal on the vertices $V_n$}\label{ssec3}
This time we perform the sum over $\si^\re$ in \eqref{eq:pfss}.
Let $g=\langle u,v\rangle\in E_n$ be an edge with weight $\eps_g$.  We have
\begin{align}
\sum_{\si_g=\pm 1}\left(1+\eps_g\si_u\si_g\right)\left(1+\eps_g\si_v\si_g\right)
&=2\left(1+\eps_g^2\si_u\si_v\right),\label{ed}\\
\sum_{\si_g=\pm 1}\si_g\left(1+\eps_g\si_u\si_g\right)\left(1+\eps_g\si_v\si_g\right)
&=2\eps_g(\si_u+\si_v).\label{ed2}
\end{align}
By \eqref{eq:pfss} and \eqref{ed},
\begin{equation}\label{eq:pfss2}
\ZNI= 2^{|E_n|}\sum_{\si^\rv \in \Si_n^\rv} \prod_{g=\langle u,v\rangle\in E_n}
\left(1+\eps_g^2\si_u\si_v\right).
\end{equation}
By \eqref{eq:tanh}, this is the partition function of an Ising model on $\HH_n$ 
with (possibly complex) weights.

Let $e=\langle u,v\rangle$, $f=\langle x,y\rangle$ be distinct edges in $E_n$. 
Motivated by Section \ref{ssec2} and the discussion of the two-edge correlation 
$\langle \si_e\si_f\rangle_n$ of the \ot model,
we define
\begin{equation}\label{eq:pfss3}
\si(e,f) = \frac1{\ZNI } 
\sum_{\si^\re\in\Si_n^\re} \sum_{\si^\rv \in \Si_n^\rv} \si_e\si_f\prod_{v\in V_n}(1+\eps_a\si_v\si_{v,a})(1+\eps_b\si_v\si_{v,b})(1+\eps_c\si_v\si_{v,c}).
\end{equation}
By \eqref{ed2}--\eqref{eq:pfss2}, this equals
\begin{align}\label{eq:corr4}
&\frac1{\ZNI }2^{|E_n|}\sum_{\si^\rv \in \Si_n^\rv} 
\frac{\eps_e(\si_u+\si_v)\eps_f(\si_x+\si_y)}{\left(1+\eps_e^2\si_u\si_v\right)\left(1+\eps_f^2\si_x\si_y\right)}
\prod_{g=\langle u,v\rangle\in E_n}\left(1+\eps_g^2\si_u\si_v\right)\\
&\hskip2cm = \sum_{\si^\rv\in\Si_n^\rv} D_{e,f}(\si^\rv) w(\si^v)
\biggl/ \sum_{\si^\rv\in\Si_n^\rv}  w(\si^\rv),
\nonumber
\end{align}
where
\begin{align*}
w(\si^\rv) &= \prod_{g=\langle u,v\rangle\in E_n}\left(1+\eps_g^2\si_u\si_v\right),\\
D_{e,f}(\si^\rv) &= 
\frac{\eps_e(\si_u+\si_v)\eps_f(\si_x+\si_y)}
{\left(1+\eps_e^2\si_u\si_v\right)\left(1+\eps_f^2\si_x\si_y\right)}, \qq 
e=\langle u,v\rangle, \ f=\langle x,y\rangle.
\end{align*}
We interpret $D_{e,f}(\si^\rv)$ as $0$ when its denominator is $0$.
Since $\si_1+\si_2=0$ when $\si_1\si_2=-1$,  we may write
\begin{equation}
D_{e,f}(\si^\rv) = 
\frac{\eps_e(\si_u+\si_v)\eps_f(\si_x+\si_y)}
{(1+\eps_e^2)(1+\eps_f^2)}.\label{eecis0}
\end{equation}

If the weights $w(\si^\rv)$ are positive (which they are not in general), 
the ratio on the right side of \eqref{eq:corr4} may be interpreted as an expectation. 
This observation will be used in Section \ref{sec:pf-1}.

By inspection of \eqref{eq:pfss2}, if $\eps_g^2=\pm 1$ 
for some $g\in\{a,b,c\}$, then zero mass
is placed on configurations $\si$ for which there exists an edge $\langle u,v\rangle$
of type $g$ with $\si_u\si_v=\mp 1$. 
We turn to the special case of \eqref{ce} and \eqref{abc} with $ABC\ne 0$.
Then 
\begin{equation}\label{eq:eps=}
\eps_a^2=\begin{cases}
-1 &\text{if and only if } a^2=b^2+c^2,\\
1 &\text{if and only if } bc=0,\\
\end{cases}
\end{equation}
and similarly for $\eps_b$, $\eps_c$.
Note in this case that $\si(e,f)=\langle \si_e\si_f\rangle_n$, the two-edge function
for the associated \ot model.

\subsection{Proof of Theorem \ref{thm:main}(b)}\label{sec:pf-1}
Let $e=\langle u,v\rangle$, $f=\langle x,y\rangle$ be distinct edges in $E_n$ such that
$u$, $x$ are white and $v$, $y$ are black. By \eqref{eq:pfss3}--\eqref{eecis0},
\begin{equation}
\langle\si_e\si_f\rangle_n=
\big\langle D_{e,f}(\si^\rv)\big\rangle^\I_n,
\label{eeci}
\end{equation}
where $\langle \cdot \rangle_n^\I$ denotes expectation in the  Ising model
of \eqref{eq:pfss2}. 
Recall that this Ising model may have complex weights.

By \cite[Cor.\ 2.5]{Lis} and known results for the Kac--Ward operator
(see \cite{CCK,KLM,LisF}), we have that
$\langle\si_e\si_f\rangle:=\lim_{n\to\oo}\langle\si_e\si_f\rangle_n\to 0$
exponentially fast as $|e-f|\to\oo$, so long as the three acute angles with tangents
$|\eps_g^2|$, $g=a,b,c$, have sum $\th$ satisfying $\th<\frac12\pi$.

It suffices to assume that 
$a \ge b,c >0$ and $\sqrt a < \sqrt b + \sqrt c$. Suppose first that, in addition,
\begin{equation}\label{eq:not0}
a \ne b+c.
\end{equation}
Let $A$, $B$, $C$ be given by 
\eqref{abc} and \eqref{ce}, so that $A \in (-1,1)\setminus\{0\}$ and $B,C<0$. 
Note that the $\eps_g$ of \eqref{ce} are purely imaginary if $A<0$, 
and real otherwise. Now,
\begin{equation}\label{eq:iso}
\tan \th=|ABC|\frac{A^{-2}+B^{-2}+C^{-2}}{1-A^2-B^2-C^2},
\end{equation}
which is finite under \eqref{eq:not0} and strictly positive if 
\begin{equation}\label{eq:sum2}
A^2+B^2+C^2<1.
\end{equation}
Using \eqref{abc}, it is a short calculation to see that 
\eqref{eq:sum2} holds if 
$a^2+b^2+c^2-2ab-2bc-2ca < 0$, which is indeed valid when
$\sqrt a < \sqrt b + \sqrt c$. (See also the proof of Proposition \ref{prop:Pzero}.)
This establishes \eqref{eq:expfast} subject to \eqref{eq:not0}.

Suppose finally that $a=b+c$, so that $A=0$ and $B,C<0$, $B+C=-1$.
It is useful to represent the 1-2 model as a polygon model, via its high-temperature expansion.
As explained in \cite{GL7}, the two-edge function satisfies 
$$
\langle\si_e\si_f\rangle_n=\frac{Z_{n,e\lra f}}{Z_n(P)},
$$
where $Z_{n, e\lra f}$ and $Z_n(P)$ are given at \cite[eqns (2.3), (2.7)]{GL7}
with
$$
\eps_b\eps_c=A, \q \eps_a\eps_c=B, \q \eps_a\eps_b=C.
$$

For a polygon configuration $\pi$ (that is, a set of edges such that every vertex
has even degree), a vertex of $\HH_n$ is said to be of type $ab$ if it is incident to
two edges with types $a$ and $b$ (and similarly for $ac$ and $bc$).
Since each vertex in the polygon model has even degree, and $A=0$, no vertex of $\HH_n$
has type $bc$. Therefore, any polygon configuration with non-zero weight in $Z_n(P)$
is a disjoint union of cycles comprising 
$ac$-type and $ab$-type vertices. 
The vertices on such a cycle form consecutive pairs with the same type,
and each such pair contributes weight either $B^2$ or $C^2$. 
It follows that $Z_n(P)$ is a sum of positive weights. 

Suppose first that $e$ and $f$ are $a$-type (horizontal) edges.
Let $\pi'$ be a path between the midpoints
of  $e$ and $f$ that contributes a non-zero weight to $Z_{n,e\lra f}$,
and let $h$ be the number of its $a$-type edges (with each $a$-type
half-edge contributing $\frac12$). Then $\pi'$
contains exactly $2h$ vertices of $\HH_n$, 
which appear in consecutive pairs with the same type (either $ab$ or $ac$).
The product of the weights of the vertices of $\pi'$ is $B^{2v}C^{2(h-v)}$ ($>0$), 
where
$v$ is the number of consecutive pairs with type $ac$. We denote
by $T(h,v)$ the set of all such $\pi'$.

Since  the removal of $\pi'$ gives a configuration contributing to $Z_n(P)$,
and in addition $D:=B^2+C^2<1$,
\begin{align*}
0\le \langle\si_e\si_f\rangle_n &= \frac{Z_{n,e\lra f}}{Z_n(P)}
\le\frac1{Z_n(P)} \sum_{h,v} \,\sum_{\pi'\in T(h,v)} B^{2v}C^{2(h-v)}Z_n(P)\\
&\le  
2\sum_{h=H}^\oo \sum_{v=0}^h \binom hv B^{2v}C^{2(h-v)}
= \frac{2 D^H}{1-D},
\end{align*}
where $H=\inf\{h: T(h,v)\ne\es\text{ for some }v\}$. 
Since $H \ge c|e-f|$ for some $c>0$, the claim follows for horizontal $e$, $f$. 

If either $e$ or $f$ is not horizontal, an extra term appears 
at one or both of the ends of $\pi'$, and such a term contributes
a factor bounded by $|\max\{B,C\}|<1$.

\begin{remark}\label{rem:subcrit}
The conclusion \eqref{eq:expfast} of Theorem \ref{thm:main}(b) 
may be proved as follows 
subject to the more restrictive condition $a^2<b^2+c^2$.
Under this condition, we have that $A,B,C < 0$.
The graph $\HH _n$ is bipartite with vertex-classes coloured black and white
(see the discussion around Figure \ref{fig:sign}).
We  now reverse the signs of the spins of black vertices, thereby obtaining a ferromagnetic
Ising model. It is easily checked that this is a  high-temperature model (as in \eqref{eq:iso}),
and it follows that $\langle \si_e\si_f\rangle\to 0$ exponentially fast as
$|e-f|\to\oo$.   
\end{remark}

\begin{remark}\label{rem:alt}
A further proof of parts of Theorem \ref{thm:main} is presented in Section \ref{sec:eecad}.
This proof is based on `dimer' rather than `Ising' methods, and may be extended
to periodic 1-2 models which appear to be currently beyond the Kac--Ward techniques 
of \cite{Lis}. See Section \ref{sec:periodic}.
\end{remark}

\section{Dimer representation of the \ot model}
\label{sec:dimer}

\subsection{The decorated dimer model}\label{ssec:dec}

Let $\HnD=(V_{n,\De},E_{n,\De})$ 
be the decorated toroidal graph derived from $\HH_n$ and illustrated 
on the right of Figure \ref{fig:12con}.
It was shown in \cite{ZL2} that there is a correspondence between 
\ot configurations on $\HH_n$ and dimer configurations on $\HnD$. 
This correspondence is summarized in the figure caption, and a more detailed description follows.

\begin{figure}[htbp]
\centerline{\includegraphics[width=0.9\textwidth]{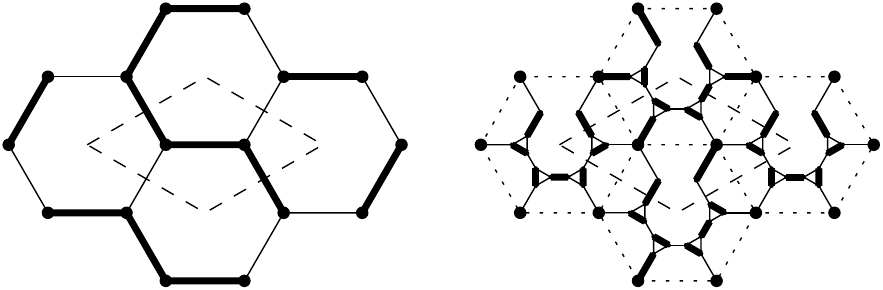}}
\caption{Part of a \ot configuration on $\HH_n$, and the corresponding dimer (sub)configuration
on $\HH_{n,\De}$.
When two edges with a common vertex of $\HH_n$ have the same state in the \ot model,
the corresponding `bisector edge' is present in the dimer configuration.
The states of the bisector edges determine the dimer configuration on the rest of $\HH_{n,\De}$.
The edges of $\HH_{n,\De}$ are allocated weights consistently with the 
\ot weights of Figure \ref{fig:sign}. The central lozenge of the right-hand figure is expanded in Figure \ref{fig:ofd0}.}
\label{fig:12con}
\end{figure}

Let $\si$ be a \ot configuration on $\HH_n$, and let $v \in V_n$ ($\subseteq V_{n,\De}$). 
The vertex $v$ has three incident edges in $\HH_{n,\Delta}$, which are bisectors of the three angles 
of $\HH_n$ at $v$. Such a \emph{bisector edge} is present in the dimer configuration 
on $\HnD$ if and only if the two 
edges of the corresponding angle have the same $\si$-state, that is, either both or neither are present. 
The states of the bisector edges determine the dimer configuration on the entire $\HH_{n,\Delta}$. 
Note that the \ot configurations $\si$ and $-\si$ generate the same dimer configuration, denoted $D_\si$.

To the edges of $\HnD$ we allocate weights as follows: edge $e=\langle i,j\rangle$ 
is allocated weight $w_{i,j}$ where
\begin{equation}\label{eq:weights}
w_{i,j} = 
\begin{cases} a &\text{if $e$ is a horizontal bisector edge,}\\
 b &\text{if $e$ is a NW bisector edge,}\\ 
 c &\text{if $e$ is a NE bisector edge,}\\
1 &\text{otherwise}.
\end{cases}
\end{equation}
The weight of a dimer configuration
is the product of the weights of present edges.

To each \ot configuration $\si$ on $\HH_n$, there corresponds 
thus a unique dimer configuration on $\HnD$.
The converse is more complicated, and we preface the following discussion 
with the introduction of the planar graph $\HH_n'$, derived from $\HH_n$
by a process of `unwrapping'  the torus.  

Let 
$\HH_n'$ be the planar graph obtained from $\HH_n$ by cutting through 
the two homology cycles $\g_x$ and $\g_y$ of the torus, as illustrated in Figure \ref{hex0}. 
That is, $\HH_n'$ may be viewed as the set of edges that
intersect the region marked in Figure \ref{hex0} (in which $n=4$ and the central edge
is labelled $\langle u,v \rangle$). 
We may consider $\HH_n'$ as a \lq partial-graph' 
$\HH'_n=(V_n,\wt{E}_n,H_n)$, where $V_n$ is the vertex set, $\wt{E}_n$ is 
the `internal' edge set, and $H_n$ is the set of half-edges having one 
endpoint in $V_n$ and one outside $V_n$. 
We write $H_x^1$ and $H_x^2$ (\resp,
$H_y^1$, $H_y^2$) for the sets of half-edges
that cross the upper left and lower right sides 
(\resp, upper right and lower left sides) of the diamond of Figure \ref{hex0}.
Let $H_u=H_u^1\cup H_u^2$ for $u=x,y$. 
 
\begin{figure}[htbp]
\centering
\includegraphics*[width=0.6\hsize]{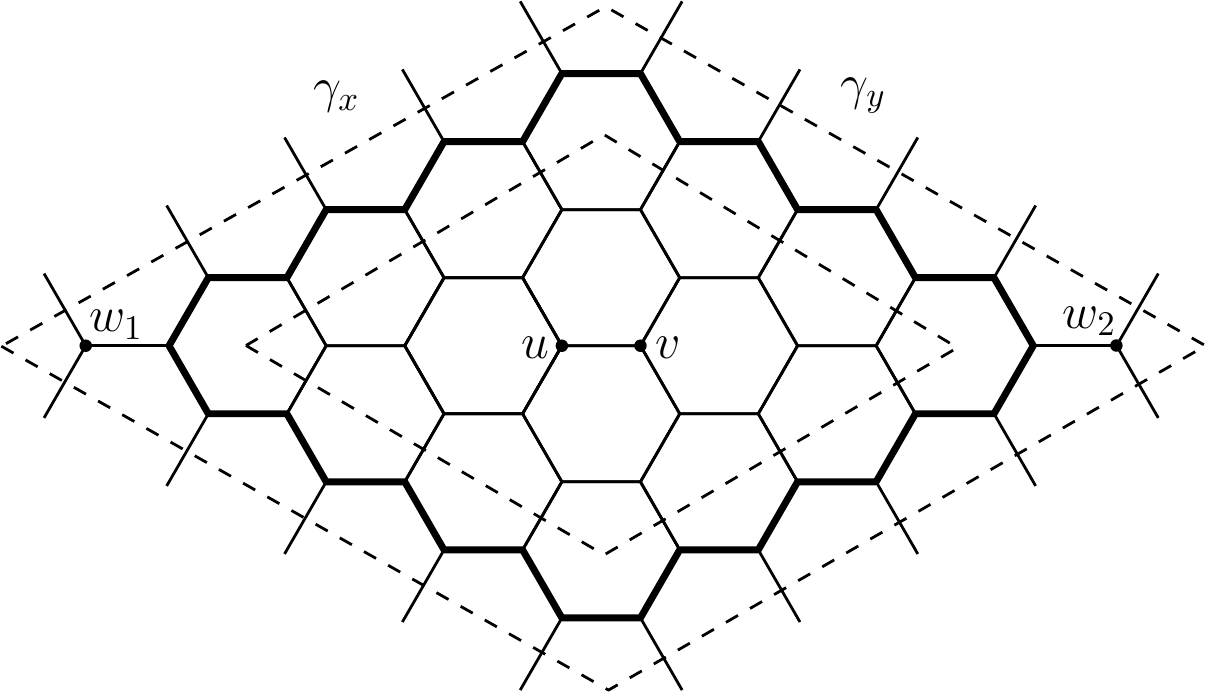}
\caption{The $n\times n$ `diamond' of $\HH$, with $n=4$. 
The region $\HH'_n$ comprises all edges and half-edges that intersect the larger diamond.
The annulus between the given boundaries comprises a cycle $C_n$
(drawn above in bold),
and two further edges incident with the $w_i$.}\label{hex0}
\end{figure}

A \ot  configuration on $\HH'_n$ is a subset of edges and half-edges
such that, for $v \in V_n$,  the total number of edges and half-edges that are incident
to $v$ is either 1 or 2. 
It is explained in \cite[p.\ 4]{ZL2} that 
dimer configurations on $\HnD$ are in one-to-two correspondence to \ot  
configurations on $\HH_n'$ satisfying any of the following  (pairwise exclusive) conditions:
\begin{enumerate}
\item [(ss)] for  $e\in H_x\cup H_y$, the two corresponding half-edges $e^1\in H_x^1\cup H_y^1$, 
$e^2\in H_x^2\cup H_y^2$  have the same state (either both are present 
or neither is present); 
\item [(os)] for  $e\in H_x$, the two corresponding half-edges $e^1\in H_x^1$, $e^2\in H_x^2$ 
have the opposite states (exactly one of them is present); for $e\in H_y$, 
the two corresponding half-edges $e^1\in H_y^1$, $e^2\in H_y^2$ have the 
same state;
\item [(so)] for  $e\in H_x$, the two corresponding half-edges $e^1\in H_x^1$, $e^2\in H_x^2$ 
have the same state; for  $e\in H_y$, the two corresponding half-edges 
$e^1\in H_y^1$, $e^2\in H_y^2$ have the opposite states;
\item [(oo)] for  $e\in H_x\cup H_y$, the two corresponding half-edges $e^1\in H_x^1\cup H_y^1$, 
$e^2\in H_x^2\cup H_y^2$ have the opposite states.
\end{enumerate}
We refer to the above as the \emph{mixed boundary condition} on $\HH_n'$.

The above mixed boundary condition is more permissive than the 
periodic condition that gives rise to \ot configurations on the toroidal graph $\HH_n$,
although the difference turns out to be invisible in the infinite-volume limit
(see Theorem \ref{es}).

\subsection{The spectral curve of the dimer model}\label{ssec:spc}
 We turn now to the spectral curve of the above weighted dimer model on $\HnD$. 
The reader is referred to \cite{ZL-sc}\ for relevant background, and to \cite[Sect.\ 3]{ZL2}
for further details of the following summary.

The fundamental domain of $\HnD$ is the central
lozenge of Figure \ref{fig:12con}, as expanded in Figure \ref{fig:ofd0}. The edges
of $\HnD$ are oriented as in the latter figure. It is easily checked that this
orientation is \lq clockwise odd', in the sense that any face of $\HnD$, 
when traversed clockwise, contains
an odd number of edges oriented in the corresponding direction.
The fundamental domain has $16$ vertices, and its weighted adjacency matrix
(or `Kasteleyn matrix') is the  $16\times 16$ matrix $\mathbf B=(b_{i,j})$ with
$$
b_{i,j} = \begin{cases} w_{i,j} &\text{if $\langle i,j\rangle$ is oriented from $i$ to $j$},\\
 -w_{i,j} &\text{if $\langle i,j\rangle$ is oriented from $j$ to $i$},\\
 0 &\text{if there is no edge between $i$ and $j$},
\end{cases}
$$
where $w_{i,j}$ is given by \eqref{eq:weights}.
From $\mathbf B$ we obtain a \emph{modified} adjacency (or `Kasteleyn')
matrix $\mathbf B(z,w)$ as follows.

\begin{figure}[hbtp]
\centerline{
\includegraphics[width=0.6\textwidth]{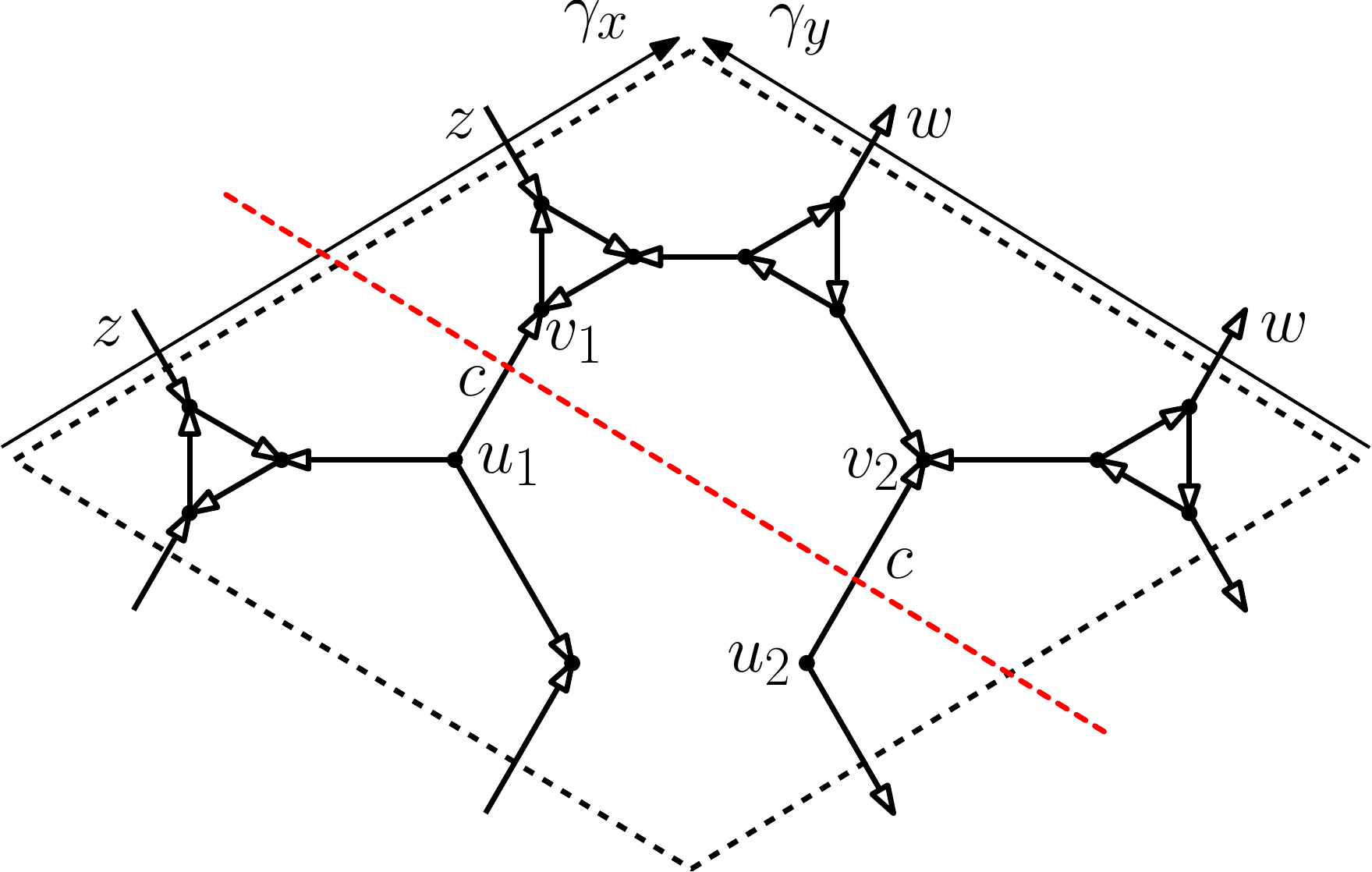}}
\caption{A single fundamental domain of the decorated graph $\HnD$
obtained from the central lozenge of Figure \ref{fig:12con}.
See that figure for an illustration of the relationship between this fundamental
domain and the original hexagonal lattice $\HH$. Note  the homology cycles
$\g_x$, $\g_y$ of the torus, and also the two weight-$c$ edges crossed by the central dashed line.}
\label{fd}\label{fig:ofd0}
\end{figure}

We may consider the graph of Figure \ref{fig:ofd0} as being embedded in a torus, that is, 
we identify the upper left boundary and the lower right boundary, and also the 
upper right boundary and the lower left boundary, as illustrated in the figure by dashed lines. 

Let $w,z \in\CC$ be non-zero.
We orient each of the four boundaries of Figure \ref{fig:ofd0} (denoted by dashed lines) from their
lower endpoint to their upper endpoint.  The `left' and `right' of an oriented 
portion of a boundary are as viewed by a person traversing in the given direction.

Each edge $\langle u,v\rangle$ crossing a  boundary corresponds to two entries in the 
weighted adjacency matrix, indexed $(u,v)$ and $(v,u)$. If the edge starting from $u$ 
and ending at $v$ crosses an  upper-left/lower-right boundary from  left to  right 
(\resp, from right to left), we modify the adjacency matrix by multiplying the entry 
$(u,v)$ by $z$ (\resp, $z^{-1}$). If the edge starting from $u$ and ending at $v$ 
crosses an upper-right/lower-left  boundary from  left to  right (\resp, 
from  right to  left), in the modified adjacency matrix, we multiply the entry by 
$w$ (\resp, $w^{-1}$). We modify the entry $(v,u)$ in the same way.
The ensuing matrix is denoted $\mathbf B(z,w)$,
for a definitive expression of which, the reader
is referred to \cite[Sect.\ 3]{ZL2}.
 
The \emph{characteristic polynomial} is given (using Mathematica or otherwise) by
\begin{align} \label{pzw}
P(z,w)=\det \mathbf B(z,w) = f(a,b,c;w,z),
\end{align}
where 
\begin{align*}
f(a,b,c;w,z)&=
a^4+b^4+c^4+6a^2b^2+6a^2c^2+6b^2c^2-2ab\left(z+\frac{1}{z}\right)\left(a^2+b^2-c^2\right)\\
 &\qq-2ac\left(w+\frac{1}{w}\right)\left(a^2+c^2-b^2\right)-2bc\left(\frac{z}{w}+\frac{w}{z}\right)
 \left(b^2+c^2-a^2\right).
 \end{align*}

The \emph{spectral curve} is the zero locus of the characteristic polynomial, that
is, the set of roots of $P(z,w)=0$. It is proved in  \cite[Lemma 3.2]{ZL2} 
that the intersection 
of $P(z,w)=0$ with the unit torus $\TT^2$ is either empty or a single real 
point $(1,1)$. Moreover, in the situation when $P(1,1)=0$,
the zero $(1,1)$ has multiplicity $2$.
It will be important later to identify the conditions under which $P(1,1)=0$.

\begin{proposition}\label{prop:Pzero} 
Let $a,b>0$ and $c\geq 0$. 
\begin{letlist}
\item  If any of the following hold,
$$
\mathrm{(i)}\q\sqrt{a}=\sqrt{b}+\sqrt{c},\q
\mathrm{(ii)}\q\sqrt{b}=\sqrt{c}+\sqrt{a},\q
\mathrm{(iii)}\q\sqrt{c}=\sqrt{a}+\sqrt{b},
$$
the curve $P(z,w)=0$ intersects 
 the unit torus $\TT^2=\{(z,w):|z|=1,|w|=1\}$ at the unique point $(1,1)$.
 \item If none of {\rm (i)--(iii)} hold, the curve does not intersect the unit torus.
\end{letlist}
\end{proposition}

\begin{proof}
The intersection of $P(z,w)=0$ with $\TT^2$ can only be either empty or a single point $(1,1)$,
by \cite[Lemma 3.2]{ZL2}. Moreover, 
since
\begin{equation}\label{eq:303}
f(a,b,c;1,1) = (a^2+b^2+c^2-2ab-2bc-2ca)^2,
\end{equation}
we have that $f(a,b,c;1,1)=0$ if and only if $\sqrt a \pm \sqrt b \pm \sqrt c=0$.
\end{proof}

We note for future use that
\begin{equation}\label{eq:304}
P(1,1) = f(a,b,c;1,1)=\tfrac14\bigl[(A^2+B^2+C^2-1)(a+b+c)^2\bigr]^2,
\end{equation}
where $A$, $B$, $C$ are as in \eqref{abc}.
 
\section{Infinite-volume limits}\label{sec:free}

This paper is directed primarily at the asymptotic behaviour of the two-edge 
correlation function of the \ot model, rather than at the existence and multiplicity
of infinite-volume measures. Partial results in the latter direction are reported in this 
section. 
In Section \ref{ssec:inf}, the weak limit of the toroidal \ot measure is
proved via a relationship with the dimer model on a decorated graph.
In Section \ref{ssec:Gibbs} we prove the non-uniqueness of
Gibbs measures for the `low temperature' \ot model. 
The existence of the infinite-volume free energy is proved in Section \ref{ssec:free}.

\subsection{Toroidal limit measure}\label{ssec:inf}

The \ot model may be studied via the dimer representation of Section \ref{sec:dimer}.
The dimer convergence theorem
of \cite{ZL2} is as follows. 

\begin{theorem}\label{thm:dimer}\cite[Prop.\ 3.3]{ZL2}
Consider the dimer measure $\de_{n,\De}$  on $\HH_{n,\Delta}$ 
with parameters $a, b, c>0$.
The limit measure $\de_\De:=\lim_{n\to\oo} \de_{n,\De}$ exists and is 
translation-invariant and ergodic.
\end{theorem}

Let $\mu_n^\m$ (\resp, $\mu_n$)
be the \ot probability measure on $\HH_n'$  (\resp, on the toroidal
$\HH_n$) with 
parameters $a$, $b$, $c$ and mixed boundary condition. 
By the results of \cite{ZL2} and the invariance of $\mu_n^\m$ under sign changes,
\begin{equation}
\mu_n^\m(\si)=\mu_n^\m(-\si)=\tfrac{1}{2}\de_{n,\De}(D_\si),
\label{sm}
\end{equation}
where $D_\si$ is the dimer configuration on $\HnD$
corresponding to the \ot configuration $\si$ on $\HH_n'$.
Since the topology of weak convergence may be given in terms of finite-dimensional cylinder events,
the weak convergence $\de_{n,\De}\to \de_\De$ entails the weak convergence of $\mu^\m_n$ to 
some probability measure $\mu^\m$ on $\HH$. By Theorem \ref{thm:dimer},
$\mu^\m$ is translation-invariant. It is noted at \cite[p.\ 17]{ZL2} that the ergodicity of
$\de_\De$ does not imply that of $\mu^\m$, 
and indeed there exist parameter values for which $\mu^\m$ is not ergodic, by the result of \cite[Thm 4.9]{ZL2}.
 
\begin{theorem}\label{es} Let $a,b,c>0$. The limit
$\mo:=\lim_{n\to\oo} \mu_n$ exists and satisfies $\mo=\mu^\m$.
In  particular, for edges $e$, $f$ of $\HH$,
the limit 
\begin{equation}\label{eq:limit}
\langle \si_e\si_f\rangle :=\lim_{n\to\oo}\langle \si_e\si_f\rangle_n
\end{equation}
exists.
\end{theorem}

\begin{proof} 
Let $\Om_{n,\De}$ be the sample space of the dimer model  on $\HH_{n,\Delta}$. 
Let $\de^\e_n$ be the probability 
measure of the dimer model on $\HH_{n,\Delta}$ on the subspace $\Om^\e_{n,\De}$
of configurations with the property that, along each of the two zigzag paths of $\HH$ that
are neighbouring and parallel to $\g_x$ and $\g_y$, there are an even number of present 
bisector edges.

As explained above (see also \cite{ZL2}),
elements of $\Om_{n,\De}$ correspond to \ot model configurations on $\HH_n'$ with 
the mixed boundary condition, and of $\Om^\e_{n,\De}$ to \ot model configurations on 
the toroidal graph $\HH_n$.  We show next that
\begin{equation}\label{eq:deconv}
\de_{n,\De}^\e \to \de_\De,
\end{equation}
where  $\de_\De:=\lim_{n\to\oo}\de_{n,\De}$ is given in Theorem \ref{thm:dimer}.

Let $Z_{n,\De}$ (\resp, $Z_{n,\De}^\e$) be the partition function of $\Om_{n,\De}$ 
(\resp, $\Om_{n,\De}^\e$), and let $K_n(z,w)$ be the modified Kasteleyn matrix of $\HH_{n,\Delta}$
(see \cite{ZL2} and Section \ref{ssec:spc}).
As explained in \cite[Sect.\ 4B]{SB08}, for $z,w\in\{-1,1\}$,  
$\Pf K_n(z,w)$ is a linear combination of partition functions of dimer 
configurations of four different classes, depending on the parity of the present edges 
along the two zigag paths
winding around the torus. In particular, by \cite[Table 1, Sect.\ 4B]{SB08}, when $n$ is even, 
 \begin{equation*}
 Z_{n,\De}^\e=\tfrac{1}{4}\Bigl[\Pf K_n(1,1)+\Pf K_n(-1,1)+\Pf K_n(1,-1)+\Pf K_n(-1,-1)\Bigr].
 \end{equation*}
 Let $e_i=\langle u_i,v_i \rangle$, $1\le i \le k$, be edges of $\HH_{n,\Delta}$, and let
 $M(e_1,\dots,e_k)$ be the event that every $e_i$ is occupied by a dimer. 
 Let $w_{i}>0$ be the edge weight of $e_i$. Then 
 \begin{align*}
 &\de^\e_n(M(e_1,\dots,e_k))\\
 &\hskip1cm =\prod_{i=1}^{k}w_{e_i}\left|\frac{\Pf \wh K_n(1,1)+\Pf \wh K_n(-1,1)+
 \Pf \wh K_n(1,-1)+\Pf \wh K_n(-1,-1)}{\Pf K_n(1,1)+\Pf K_n(-1,1)+\Pf K_n(1,-1)+\Pf K_n(-1,-1)}\right|,
 \end{align*}
 where $\wh K_n$ is the submatrix of $K_n$ obtained by removing rows and columns indexed by 
 $u_1,v_1,\dots ,u_k,v_k$.
As in \cite[Thm 4]{BdeT10}, 
\begin{align}\label{eq:66}
& \de_{n,\De}(M(e_1,\dots,e_k))\\
&\hskip1cm =\prod_{i=1}^{k}w_{e_i}
\left|\frac{-\Pf \wh K_n(1,1)+\Pf \wh K_n(-1,1)+\Pf \wh K_n(1,-1)+\Pf \wh K_n(-1,-1)}{-\Pf K_n(1,1)+\Pf K_n(-1,1)+\Pf K_n(1,-1)+\Pf K_n(-1,-1)}\right|.\nonumber
\end{align}
As in the proof of \cite[Thm 6]{BdeT10},
$\de^\e_{n,\De}(M(e_1,\dots,e_k))$ and
$\de_{n,\De}(M(e_1,\dots,e_k))$ converge as $n \to\oo$ to the same complex integral.
Since the events $M(e_1,\dots,e_k)$ generate the product $\si$-field, we deduce \eqref{eq:deconv}. 

Finally, we deduce the claim of the theorem.
An \emph{even} (\resp, \emph{odd}) \emph{correlation function} is an 
expectation of the form $\langle \si_A \rangle$, with $\si_A=\prod_{e\in A}\si_e$ 
where $A$ is a finite set of edges of $\HH$ with even (\resp, odd) cardinality.
In order that $\mu_n\to \mu^\m$, it suffices that the correlation
functions of $\mu_n$ and $\mu_n^\m$ have the same limit. By invariance under sign change, the
odd correlation functions equal $0$. 

The relationship between a \ot measure $\mu$ and the corresponding dimer measure $\de$ is as follows. 
Let $e_1,\dots ,e_k$ be bisector edges of $\HH_{n,\De}$,  and let $S(e_1,\dots,e_k)$ 
be the event that every $e_i$ separates two edges of $\HH_n$ with the same \ot state.
Using the correspondence between \ot  and dimer configurations,
\begin{equation*}
\mu(S(e_1,\dots,e_k))=\de(M(e_1,\dots,e_k)).
\end{equation*}

Let $k \ge 1$, let $e_1,e_2,\dots,e_{2k}$ be distinct edges of $\HH_n$, and write
$\si_i=\si_{e_i}$. Then
\begin{equation}\label{eq:couple}
\langle \si_1\cdots \si_{2k}
 \rangle_\mu=1-2\mu(\si_1\cdots\si_{2k}=-1).
\end{equation}
For $i=2,3,\dots, 2k$, let $\pi_i$ be a self-avoiding path between the midpoints of $e_1$ and $e_i$ 
comprising edges of $\HH_n$ and two half-edges, and 
let $\sA_i$ be the event that the number of \emph{absent} 
bisector edges encountered along $\pi_i$ 
is odd. As we move along $\pi_i$ in the \ot model, the edge-state changes at
a given vertex if and only if
the corresponding bisector edge is absent.
Therefore, $\si_1\si_i=-1$ if and only if $\sA_i$ occurs, so that
\begin{equation}\label{eq:corr5}
\langle \si_1\si_i\rangle_\mu= \mu(\ol\sA_i)-\mu(\sA_i).
\end{equation}

Let $\sA$ be the event that the set $I=\{i: \sA_i \text{ occurs}\}$ has odd
cardinality. Since $I=\{i: \si_1\ne \si_i\}$, we have that
 \begin{equation}\label{eq:couple2}
\mu(\si_1\cdots\si_{2k}=-1)=\de(\sA).
\end{equation}

We return to the measures $\mu_n$ and $\mu_n^\m$.
By \eqref{eq:deconv}--\eqref{eq:couple2}, the even correlation functions of $\mu_n$ and $\mu_n^\m$ 
are convergent as $n\to\oo$, with  equal limits. It follows that $\mu_n\to \mo$
where $\mo=\mu^\m$.
\end{proof}

\subsection{Non-uniqueness of Gibbs measures}\label{ssec:Gibbs}
We show the existence of at least two Gibbs measures (that is, `phase coexistence')
for the `low temperature' \ot model on $\HH$. 
Let $\Si$ be the set of \ot configurations on the infinite lattice $\HH$,
and let $\sG=\sG(a,b,c)$ be the set of probability measures on $\Si$ that satisfy
the appropriate DLR condition. (We omit the details of DLR measures here,
instead referring the reader to the related discussions of \cite[Sect.\ 2.3]{MB09} 
and \cite[Sect.\ 4.4]{G-RCM}.)
Since $\Si$ is compact, by Prohorov's theorem \cite[Sect.\ 1.5]{Bill}, 
every sequence of probability measures on $\Si$ has
a convergent subsequence. It may be shown that any weak limit of finite-volume \ot measures 
lies in $\sG$, and hence  $\sG\ne\es$. 

\begin{theorem}\label{pc}
Let $a\geq b >0$. For almost 
every $c$ satisfying either $0<\sqrt c<\sqrt a-\sqrt b$  or
$\sqrt c > \sqrt a+ \sqrt b$, we have that
$|\sG|\ge 2$.
\end{theorem}

\begin{proof} 
This proof is inspired by that of \cite[Thm 6.2]{MB09}, and it makes use
of Theorem \ref{thm:main}, the proof of which has been deferred to Section \ref{sec:morepf}.
Let $e$ be a given horizontal edge of $\HH_n$, and let $f_m$ be an edge 
satisfying \eqref{eq:condition0} such that
$\pi(e,f_m)$ has length $m$. 
By Theorem \ref{thm:main}(b) and translation invariance, 
for almost every $c$ satisfying the given inequalities, 
there exists $\al>0$ such that
\begin{equation}\label{eq:alpha}
\lim_{m\to\oo}\langle \si_e\si_{f_m} \rangle^2=\alpha^2,
\end{equation}
where $\langle \si_e\si_f\rangle$ is the limiting two-edge correlation as $n \to\oo$
(see Theorem \ref{es}).
It suffices to show that, subject to \eqref{eq:alpha}, $|\sG|\ge 2$.

By \eqref{eq:alpha}, there exists a subsequence $(m_k: k \ge 1)$ along which 
$\langle \si_e\si_{f_{m_k}}\rangle$ converges to either $\al$ or $-\al$. 
Assume the first; the proof is essentially the same in the second case.
For simplicity of notation, we shall assume that
\begin{equation*}
\lim_{m\to\oo} \langle\si_e\si_{f_{m}}\rangle = \al.
\end{equation*}

By the invariance of $\mo$ under sign change of the configuration,
\begin{align*}
\lim_{m\to\oo}\mo(\si_e=1\mid \si_{f_{m}}=1)=\tfrac12(1+\alpha),\\
\lim_{m\to\oo}\mo(\si_e=1\mid\si_{f_{m}}=-1)=\tfrac12(1-\alpha).
\end{align*}
Find $M$ such that
\begin{align*}
\mo(\si_e=1\mid \si_{f_{m}}=-1)&<\tfrac12(1-\tfrac12\al)\\
&<\tfrac12(1+\tfrac12\al)
< \mo(\si_e=1\mid\si_{f_{m}}=1),\qq m \ge M.
\end{align*}
We may find an increasing subsequence $(r_m: m \ge M)$
such that
\begin{align*}
\mu_{r_m}(\si_e=1\mid \si_{f_{m}}=-1)&<\tfrac12(1-\tfrac13\al)\\
&<\tfrac12(1+\tfrac13\al)
< \mu_{r_m}(\si_e=1\mid\si_{f_{m}}=1),\qq m \ge M.
\end{align*} 
Let $\mu^+$ (\resp, $\mu^-$) be a subsequential limit of $\mu_{r_m}(\cdot\mid\si_{f_{m}}=1)$
(\resp, $\mu_{r_m}(\cdot\mid\si_{f_{m}}=-1)$), so that
\begin{equation*}
\mu^{+}(\si_e=1)>\mu^{-}(\si_e=1).
\end{equation*}
In particular, $\mu^+\ne \mu^-$.
Since $|e-f_{m}| \to \oo$ as $m \to\oo$, the measures $\mu^\pm$ satisfy the DLR condition, 
and therefore they lie in $\sG$.
\end{proof}

\subsection{Free energy}\label{ssec:free}

A \emph{boundary condition} $\sB_n$ is a 
configuration on the half-edges $H_n$ of the planar graph $\HH_n' = (V_n, \wt E_n, H_n)$,
in the notation of Section \ref{ssec:inf}. 
Let $Z_n(a,b,c,\sB_n)$ be the partition function of the \ot model 
on $\HH'_n$ with parameters $a$, $b$, $c$ and boundary condition $\sB_n$
(as in \eqref{pft}, say). 
The free energy, for given $a$, $b$, $c$ and 
boundary conditions $(\sB_n: n\ge 1)$, is defined to be
\begin{equation}
\sF(a,b,c,(\sB_n)) := \lim_{n\to\oo}\frac{1}{|V_n|}\log Z_n(a,b,c,\sB_n)\label{fed},
\end{equation}
whenever the limit exists.

\begin{proposition}
Let $(a,b,c)\ne (0,0,0)$.
The free energy of \eqref{fed} exists and is independent of
the choice of boundary conditions $(\sB_n)$. 
Moreover, up to a smooth additive constant, it satisfies
\begin{equation}
\sF(a,b,c)=\frac{1}{4\pi^2}\iint_{[0,2\pi]^2} \log P(e^{i\theta},e^{i\phi})
\,d\theta\, d\phi,\label{fef}
\end{equation}
where $P$ is given in \eqref{pzw}.
\end{proposition}

\begin{proof}
The correspondence 
between \ot model configurations on $\HH_n'$ (with the mixed boundary condition) 
and dimer configurations on $\HnD$
was explained in Section \ref{ssec:inf}.
It follows that the free energy of that \ot model is the same as that of the corresponding 
dimer model. The expression \eqref{fef} follows for that case from a 
general argument used to compute the free energy of this dimer model, 
given that either the spectral curve does not intersect the 
unit torus, or the intersection is a unique real point of multiplicity $2$.
See  Proposition \ref{prop:Pzero}
and also \cite[Thm 3.5]{KOS06} and \cite[Thm 1]{BdeT10}.

Next we prove that the free energy of \eqref{fed} is 
independent of the choice of $(\sB_n)$. 
To this end, we consider the boxes $\HH'_n$ and $\HH'_{n-2}$ illustrated in Figure \ref{hex0}.
We claim that, for any boundary condition on $\HH'_n$ (that is, any
present/absent configuration on $H_n$) and any \ot model configuration 
on $\HH'_{n-2}$ (so that the edge-states on $\wt{E}_{n-2}\cup H_{n-2}$ are given), 
there exists a configuration on 
$\wt{E}_n\setminus (\wt{E}_{n-2}\cup H_{n-2})$  such that the
composite configuration is a \ot configuration on $\HH'_n$. 

This claim is shown as follows. Consider a given boundary condition on $H_n$
and a \ot configuration on $\HH'_{n-2}$. 
The vertex-set $V_n \setminus V_{n-2}$ forms a cycle $C_n$
with even length, together with two 
further vertices $w_1$, $w_2$ at the left and right corners, see Figure \ref{hex0}.
From $C_n$ we select a perfect matching. By considering the various possibilities,
we may see that it is always possible to allocate states to the two edges
between $C_n$ and the $w_i$ in such a way that, in the resulting composite configuration,
each $w_i$ has degree either $1$ or $2$.

Let $\sB_n^0$ be the \emph{free boundary condition}, under which
no half-edge is present. We have that
\begin{equation*}
\bigl|\log Z_n(a,b,c,\sB_n)-\log Z_{n-2}(a,b,c,
\sB^0_{n-2})\bigr|\leq |V_n\setminus V_{n-2}|K,
\end{equation*}
for some $K=K(a,b,c)>0$ and all $\sB_n$. Divide by $|V_n|$ and let $n\to\oo$ to obtain the claim.
The theorem follows on noting that the number of boundary configurations is $2^{|H_n|}$,
and $|H_n|/|V_n| \to 0$ as $n\to\oo$.
\end{proof}

\section{Proofs of Theorem \ref{thm:main}($\mathrm{a}$) and of the limit
in Theorem \ref{thm:main}($\mathrm{c}$)}
\label{sec:morepf}

The basic structure of the proof is as follows. As in Section \ref{sec:dimer}, the \ot
model may be represented as a dimer model on a certain decorated graph $\HnD$ derived from $\HH_n$. 
Subject to condition \eqref{eq:condition0}, the two-edge
correlation $\langle \si_e\si_f\rangle_n$ of the \ot model may be represented 
in terms of certain cylinder probabilities of the dimer model.  
Using the theory of dimers, these probabilities may be expressed in terms of ratios of Pfaffians
of block Toeplitz matrices, and a  similar representation follows for the infinite-volume
two-edge correlation $\langle \si_e\si_f \rangle$.  By Widom's theorem \cite{HW0,HW},
the limit $\La(a,b,c):= \lim_{|e-f|\to\oo} \langle \si_e\si_f \rangle^2$ exists, and furthermore
$\La$ is analytic except when the spectral curve intersects the unit torus.
This identifies the phases of the \ot model, and they may be identified as sub/supercritical via 
the extreme values of \eqref{eq:extreme}.

We shall refer to the following condition on the edges $e$, $f$ of $\HH$:
\begin{equation}\label{eq:condition}
\begin{aligned}
&\text{$e$ and $f$ are midpoints of two NW edges such that}\\
&\text{there exists a path $\pi=\pi(e,f)$ in $A\HH_n$ from $e$ to $f$}\\
& \text{using only horizontal and NW half-edges}.
\end{aligned}
\end{equation}
See Figure \ref{fig:lef2}. The principal step in the proof is the following.

\begin{figure}[htbp]
  \centering
\includegraphics*[scale=1]{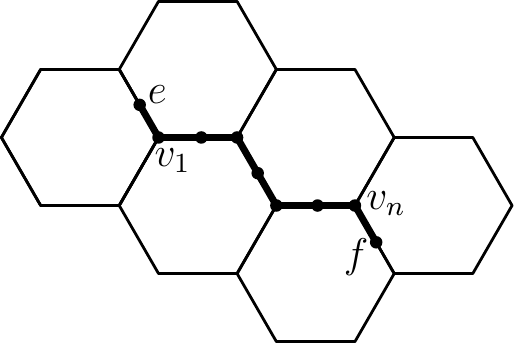}
   \caption{A path $\pi$ comprising horizontal and NW mid-edges, 
   connecting the midpoints of two NW edges $e$ and $f$.}\label{fig:lef2}
\end{figure}

\begin{theorem}\label{anly}
Let $e,f$ be two edges satisfying \eqref{eq:condition}, and let $a\geq b\geq 0$.
The limit $\La(a,b,c):=\lim_{|e-f|\to\oo}\langle \si_e\si_f\rangle^2$  
exists and is complex analytic in 
$c\geq 0$ except when $\sqrt{c}=\sqrt{a}-\sqrt{b}$ and $\sqrt{c}=\sqrt{a}+\sqrt{b}$.
\end{theorem}

Before giving the proof of Theorem \ref{anly}, we explain
how to deduce some of the claims of Theorem \ref{thm:main}(a,\,c);
the \emph{exponential rate} of convergence in part (c) is proved in 
Section \ref{sec:pf32}.

\begin{proof}[Proof of Theorem \ref{thm:main}(a,\,c), without the exponential rate]
Part (a) holds by Theorem \ref{es}.

Let $a \ge b >0$, and let $e$, $f$ satisfy \eqref{eq:condition}. 
By Theorem \ref{anly}, 
the function $\La_3(\cdot):=\La(a,b,\cdot)$ is complex analytic 
on each of the intervals 
\begin{align*}
C_1&=\bigl[0,(\sqrt{a}-\sqrt{b})^2\bigr),\\
C_2&=\bigl((\sqrt{a}-\sqrt{b})^2,(\sqrt{a}+\sqrt{b})^2\bigr),\\
C_3&= \bigl((\sqrt{a}+\sqrt{b})^2,\oo\bigr).
\end{align*}
(That is to say, for $c \in C_i$ considered as a line in the complex plane, $\La_3$ is analytic
on some open neighbourhood of $c$.)  

\begin{remark}
Note in passing that, by Remark \ref{rem:subcrit},
we have $\La_3(c)=0$ for $c$ lying in the interval
$S:=(\sqrt{a^2-b^2}, \sqrt{a^2+b^2})$. 
Since $S \subseteq C_2$
and $\La_3$ is analytic on $C_2$, we have as implied in part (b) that 
$\langle \si_e,\si_f\rangle \to 0$ as $|e-f|\to\oo$. (This  is trivial if $a=b$.)
\end{remark}

We turn to part (c). Consider first the interval $C_1$, and assume $a > b$.
Since non-trivial analytic functions have only isolated zeros, it follows that:
either $\La_3\equiv 0$ on $C_1$, or $\La_3$ is non-zero except possibly 
on a set of isolated points of 
$C_1$. By \eqref{eq:extreme}, $\La_3(0)=1$, whence the latter holds.

By \eqref{eq:extreme}, 
$\langle \si_e\si_f\rangle=1$ when $a=b=0$ and $c=1$.
Since $\La_3$ is analytic (and hence continuous) on $C_3$,  
there exists $\al>0$ such that
$\La(a,b,c)\geq \alpha$ in a small (real) neighbourhood of $(0,0,1)$. 
Since $\La$ depends only on the ratios $a:b:c$ (cf.\ \eqref{eq:ratio}), 
we deduce that, for fixed $a,b>0$ and sufficiently large $c$,
we have $\Lambda_3(c)\geq \al >0$. By Theorem \ref{anly}, $\La_3$ is analytic 
on $C_3$, and the claim holds as above.
\end{proof}

The remainder of the section is devoted to the proof of Theorem \ref{anly}.
We shall develop the notation and arguments of Section \ref{ssec:inf}.
Let $\mu_n^\m$ be the \ot measure on $\HH_n'$ with the  mixed boundary condition
of Section \ref{ssec:inf}, 
and let $\mu^\m:=\lim_{n\to\oo}\mu_n^\m$, as after Theorem \ref{thm:dimer}.
By Theorem \ref{es}, the \ot measure $\mu_n$ on $\HH_n$ satisfies 
$\mu_n \to \mo=\mu^\m$ as $n\to\oo$. 

Let $e,f$ be  edges of the hexagonal lattice $\HH$ satisfying \eqref{eq:condition}.
Let the path $\pi$ of \eqref{eq:condition} traverse a total of $2k-1$ 
edges and two half-edges, so that  
$\pi$ passes $2k$ bisector edges of the infinite decorated graph $\HH_{\Delta}$.  
We denote this set of bisector edges by 
\begin{equation}\label{eq:B}
B=\bigl\{b_i=\langle u_i,v_i\rangle: i=1,2,\dots,2k\bigr\},
\end{equation}
where $v_i\in\pi$.

Our target is  to represent $\langle\si_e\si_f \rangle$ as the Pfaffian of a truncated block 
Toeplitz matrix, as inspired by \cite[Sect.\ 4.7]{RK1}. A principal difference between \cite{RK1}
and the current work  is that, whereas bipartite graphs are considered there and the determinants
of weighted adjacency matrices are computed, in the current setting the graph is non-bipartite
and we will compute Pfaffians.

To $\HH_{\Delta}$ we assign a clockwise odd orientation as in Figure \ref{fig:ofd0}:
the figure shows  a clockwise odd orientation of $\HH_{1,\Delta}$, 
embedded in a $1\times 1$ torus, that lifts to a clockwise odd orientation of 
$\HH_{\Delta}$. As in \eqref{eq:weights}, 
a horizontal (\resp, NW, NE) bisector edge of $\HH_{\Delta}$ is 
assigned weight $a$ (\resp, $b$, $c$), and all the other edges are assigned weight 1. 
The bisector edges $g_i=\langle u_i,v_i\rangle$ are oriented in such a way that
each $g_i$ is oriented from $u_i$ to $v_i$ in this clockwise-odd orientation. 

Let $K_n$ be the Kasteleyn matrix of $\HnD$
(as in Section \ref{ssec:spc} and \cite{ZL2}), and 
let $|v|$  denote the 
index of the row and column of $K_n$  corresponding to the vertex $v$.
Assume that $|v_i|=|u_i|+1$ for $1\le i \le 2k$, and furthermore that
\begin{equation}\label{eq:order}
|u_1|<|v_1|<|u_2|<|v_2|<\dots <|u_{2k}|<|v_{2k}|.
\end{equation} 
Let $K^{-1}$ be the  
infinite matrix whose entries are the limits of the entries of $K_n^{-1}$ as $n\to\oo$.
The existence of $K^{-1}$ may be proved by
an explicit diagonalization of $K_n$ using periodicity, as in \cite[Sect.\ 7]{ckp00}
and \cite{ZL1}.

We now construct the modified Kasteleyn matrix $K_1(z,w)$ of $\HH_{1,\De}$ by multiplying the 
corresponding entries in its Kasteleyn matrix by $z$ or $z^{-1}$ (\resp, 
$w$ or $w^{-1}$), according to the manner in which the  edge crosses one of the two homology 
cycles $\g_x$, $\g_y$ indicated
in Figure \ref{fig:ofd0}. 
As remarked in Section \ref{ssec:spc}, the characteristic polynomial $P(z,w)=\det K_1(z,w)$ is the function 
$f(a,b,c;w,z)$ of \eqref{pzw}, see also \cite[Lemma 9]{ZL2}. 
The intersection of the spectral curve $P(z,w)=0$ and the unit torus $\TT^2$
is given by Proposition \ref{prop:Pzero}. 

Consider the toroidal graph $\HnD$. Let $\g_x$, $\g_y$ be homology cycles of the torus,
which 
for definiteness we take to be shortest cycles composed of unions of boundary segments of fundamental
domains as in Figure \ref{fig:ofd0}.
Let $K_n(z,w)$ be the modified Kasteleyn matrix of $\HnD$.

For $I\subseteq \{1,2,\dots,2k\}$, let $M_I$ be the event that 
every $b_i$ with 
$i\in I$ is present in the 
dimer configuration. Assume $n$ is sufficiently large that
\begin{equation}\label{eq:noint}
\text{for $1\le i \le 2k$,\q the edge $b_i$  intersects neither $\gamma_x$ nor $\gamma_y$}.
\end{equation} 
For $n$ even, as in \eqref{eq:66},
\begin{align}\label{eq:234}
&\mo(M_I)
=\lim_{n\to\oo}\mu_n(M_I)\\
&\hskip3mm=\lim_{n\to\oo} J_I
\frac{-\Pf \wh K_{n,{I}}(1,1)+\Pf \wh K_{n,{I}}(1,-1)+\Pf \wh K_{n,{I}}(-1,1)+\Pf \wh K_{n,{I}}(-1,-1)}
{-\Pf K_n(1,1)+\Pf K_n(1,-1)+\Pf K_n(-1,1)+\Pf K_n(-1,-1)},\nonumber
\end{align}
where 
\begin{equation}\label{235}
J_I= \prod_{i\in I}[K_n]_{u_i,v_i},
\end{equation}
and $\wh A_I$ denotes the submatrix of the matrix $A$ after deletion
of  rows and columns corresponding to 
$\{u_i, v_i: i\in I\}$ (see \cite[Thm 0.1]{ZL2}). 
Note that
$\mo(M_\es)=1$.

The limit of \eqref{eq:234} can be viewed as follows. 
Each monomial in the expansion of $\Pf K_n(z,w)$, 
$z,w\in\{-1,1\}$, corresponds to the product of edge-weights of a dimer configuration, 
with possibly negative sign. 
The coefficients in the linear combination of $\Pf K_n(1,1)$, $\Pf K_n(1,-1)$, 
$\Pf K_n(-1,1)$, and $\Pf K_n(-1,-1)$ are chosen in such a way that
the products of edge-weights of 
different dimer configurations correspond to monomials of the same sign. The numerator of \eqref{eq:234}
is the sum over dimer configurations containing every $b_i$, $i\in I$; 
this can be computed by the corresponding sum of monomials in the expansion 
of the denominator. Under \eqref{eq:noint},
 $[K_n]_{u_i,v_i}:= [K_n(z,w)]_{u_i,v_i}$ is independent of 
 $z,w\in\{-1,1\}$. 
Since each $b_i$ is oriented from $u_i$ to $v_i$, we have that
$[K_n]_{u_i,v_i}=c$, whence $J_I=c^{|I|}$.

\begin{lemma}\label{pfc}
Let $A$ be a $2m\times 2m$ invertible, anti-symmetric matrix, and let $L\subseteq\{1,2,\dots,2m\}$ 
be a nonempty even subset. 
Let $\wh A_L$ be the submatrix
 of $A$ obtained by deleting the rows and columns indexed by elements in $L$, and let $A_L^{-1}$ 
 be the submatrix of $A^{-1}$ with rows and columns indexed by elements in $L$. Then
\begin{equation*}
(-1)^{S(L)}\Pf \wh A_L=\Pf (A)\Pf(A_L^{-1}),
\q \mathrm{where}\ S(L)=\sum_{l\in L} l.
\end{equation*}
\end{lemma}

\begin{proof}
See, for example, \cite[Lemma A.2]{CSS}.
\end{proof}

The conclusion of Lemma \ref{pfc} holds also when $L=\es$,
subject to the convention that $\Pf(A_{\es}^{-1})=1$.

Returning to \eqref{eq:234}, take $L=\{|u_i|,|v_i|: i \in I\}$, so that
$(-1)^{S(L)}=(-1)^{|I|}$, by the choices before \eqref{eq:order}. 
When the spectral curve does not intersect the unit torus $\TT^2$, by Lemma \ref{pfc},
\begin{align*}
\lim_{n\to\oo}\frac{\Pf\wh {K}_{n,I}(z,w)}
{\Pf {K}_n(z,w)}
&=\lim_{n\to\oo}(-1)^{|I|}\Pf K^{-1}_{n,I}(z,w)\\
&=(-1)^{|I|}\Pf K^{-1}_{I},
\end{align*}
where the limit is independent of $z,w\in\{-1,1\}$; see 
\cite[Lemma 4.8]{ZL1} for a proof of the existence of the limits of the entries of 
$K_n^{-1}$. 
By \eqref{eq:234}--\eqref{235}, 
\begin{equation}
\mo(M_I)=(-c)^{|I|}\Pf K_{I}^{-1}.\label{mm}
\end{equation}

We shall make use of the following elementary lemma, the proof of which is omitted.

\begin{lemma}\label{lem:gf}
Let $S$ be a random subset of the finite nonempty set $B$. The probability generating
function (pgf) $G(x)=\EE(x^{|S|})$ satisfies
\begin{equation*}
G(1+\la) = \sum_{I \subseteq B} \la^{|I|} \PP(S \supseteq I), \qq \la \in \RR.
\end{equation*}
\end{lemma}

Let $B$ be the set of bisector edges along $\pi$ (see \eqref{eq:B}), and let
$S$ be the subset of such  edges that are present in the dimer configuration.
By \eqref{eq:corr5}, 
$$
\langle \si_e\si_f\rangle = G(-1),
$$ 
where $G$ is the pgf of $|S|$ under the measure $\mo$. By \eqref{mm} and Lemma \ref{lem:gf},
\begin{align}\label{eq:237}
\langle \si_e\si_f\rangle = \sum_{I \subseteq B} (-2)^{|I|} \mo(M_I)
= \sum_{I \subseteq B} (2c)^{|I|} \Pf K_I^{-1}.
\end{align}
This may be recognized as the Pfaffian of a certain matrix defined as follows. 

Let $Y_1(\lambda)$ be the $2\times 2$ matrix
\begin{equation*}
Y_1(\lambda)=\begin{pmatrix} 0&\lambda\\-\lambda&0 \end{pmatrix},
\end{equation*}
and let $Y_{2k}(\lambda)$ be the $4k\times 4k$ block diagonal matrix with 
diagonal $2\times 2$ blocks equal to $Y_1(\lambda)$. More precisely, 
$Y_{2k}(\lambda)$ has rows and columns indexed $u_1, v_1, u_2, v_2, \dots,
u_{2k},v_{2k}$, and
\begin{equation*}
Y_{2k}(\lambda)=\begin{pmatrix}
Y_1(\lambda)&0&\cdots&0\\
0&Y_1(\lambda)&\cdots&0\\
\vdots&\vdots&\ddots&\vdots\\
0&0&\cdots&Y_1(\lambda)
\end{pmatrix}.
\end{equation*}

\begin{lemma}\label{lem:pf2}
We have that
\begin{equation}\label{eq:236}
\langle \si_e\si_f \rangle
=\Pf[Y_{2k}(1)+2cK_{V_\pi}^{-1}],
\end{equation}
where $V_\pi=\{u_1,v_1,u_2,v_2,\dots,u_{2k},v_{2k}\}$.
\end{lemma}

\begin{proof}
It suffices by \eqref{eq:237} that 
\begin{equation}\label{eq:pf3-}
\Pf[Y_{2k}(1)+A] = \sum_{I\subseteq B}\Pf A_I,
\end{equation}
where $A=(a_{i,j})$ is a $4k\times 4k$ anti-symmetric matrix with consecutive
pairs of rows/columns indexed by the set $B=\{1,2,\dots,2k\}$, and 
$A_I$ is the submatrix of $A$ with pairs of  rows and columns 
indexed by $I\subseteq B$.

Let $G=(V,E)$ be the complete graph with vertex-set $V=\{1,2,\dots,4k\}$,
and recall that
\begin{equation}\label{eq:pfa-}
\Pf A =\sum_{\mu\in \Pi}\sgn(\pi_\mu)\prod_{\substack {(i,j)\in\mu\\ i<j}}a_{i,j},
\end{equation}
 (see \cite{Kast61,Thom}), where $\Pi$ is the set of perfect matchings of $G$, and the permutation $\pi_\mu\in S_{4k}$ is 
given by
\begin{equation}\label{eq:pfa2-}
\pi_\mu=\begin{pmatrix}
1 &2 &3 & 4 &\cdots &4k-1 &4k\\
i_i &j_1 & i_2 & j_2 &\cdots &i_{2k} &j_{2k}
\end{pmatrix}
\end{equation}
where $\mu =\{(i_r,j_r): 1\le r \le 2k\}$, $i_1<i_2<\dots < i_{2k}$, and $i_r<j_r$. 

By \eqref{eq:pfa-},
\begin{align}\label{eq:pfa+}
\Pf[Y_{2k}(1)+A] &=
\sum_{\mu\in \Pi}\sgn(\pi_\mu)\prod_{\substack {(i,j)\in\mu\\ i<j}}[Y_{2k}(1)+A]_{i,j}\\
&=\sum_{K\subseteq V}\biggl(
\sum_{\mu\in \Pi}\sgn(\pi_\mu) \prod_{\substack {(i,j)\in\mu\\i\in K,\  i<j}}[Y_{2k}(1)]_{i,j}
\prod_{\substack {(i,j)\in\mu\\ i\notin K,\ i<j}}a_{i,j}\biggr).\nonumber
\end{align}
The penultimate product is $0$ unless every $i\in K$ is odd and satisfies $(i,i+1)\in \mu$.
Therefore, with $J=\frac12(K+1)\subseteq B$,
\begin{align*}
\Pf[Y_{2k}(1)+A] 
&=\sum_{J\subseteq B}\biggl(
\sum_{\mu\in \Pi}\sgn(\pi_\mu) 
\prod_{\substack {(i,j)\in\mu,\ i<j\\ i,j\notin (2J-1)\cup (2J)}}a_{i,j}\biggr)\\
&=\sum_{J\subseteq B} \Pf A_{B\setminus J},
\end{align*}
as required for \eqref{eq:pf3-}.
\end{proof}

Recall that a matrix is \emph{Toeplitz} if every descending diagonal is constant,
and a block matrix is \emph{block Toeplitz} if  each block is Toeplitz
and every descending diagonal of blocks is constant.
Now, $Y_{2k}\left(1\right)+2cK_{V_\pi}^{-1}$ is a truncated block Toeplitz matrix 
each block of which has size $4\times 4$.
 We propose to use Widom's formula
(see Theorem \ref{wi}) to study the limit of its determinant
as $k\to\oo$.
In this limit, the matrix becomes an infinite block Toeplitz matrix $T(\psi)$ with symbol given by
\begin{equation}
\psi(z)=\frac{1}{2\pi}\int_{0}^{2\pi}\phi(z,e^{i\theta})\,d\theta\label{psz}
\end{equation}
where 
\begin{equation}
\label{pszi}
\phi(z,e^{i\th}) =Y_2(1)+ 2cK_1^{-1}(z,e^{i\th})_{(1:4)},
\end{equation}
and $A_{(1:4)}$ denotes the $4 \times 4$ submatrix of the matrix $A$ with rows and columns
indexed by $u_1,v_1,u_2,v_2$ as in Figure \ref{fig:ofd0}. 
This follows by the explicit calculation
\begin{equation}
[K^{-1}]_{u,v}=\frac{1}{4\pi^2}\int_0^{2\pi}\int_0^{2\pi}e^{i k\phi}[K_1^{-1}(e^{i\theta},e^{i\phi})]_{u,v'}\,d\theta\, d\phi,\qq u,v\in V_\pi,\label{kiuv}
\end{equation}
where $v'$ is the translation of $v$ to the same fundamental domain as $u$, 
and $k$ is the number of fundamental domains traversed in moving from $v'$ to $v$, 
with sign depending on the direction of the move. When $k\neq 0$, 
\eqref{kiuv} is the $k$th Fourier coefficient of the symbol  \eqref{psz}.  
See \cite[Sect.\ 4]{ZL1} for a similar computation.

\begin{lemma}\label{lem:lem102}
Let $z\in\CC$ with $|z|=1$.
When the spectral curve does not intersect the unit torus 
$\TT^2$, we have that $\det \psi(z)=1$.
\end{lemma}

\begin{proof}
Let $\HH_{m,n,\Delta}$ be the toroidal graph comprising $m\times n$ 
fundamental domains (a fundamental domain is drawn in Figure \ref{fig:ofd0}).
We can think of $\HH_{m,n,\Delta}$ as the quotient graph of $\HH_{\Delta}$ 
under the action of $m\ZZ\times n\ZZ$.
To $\HH_{m,n,\Delta}$ we allocate the 
clockwise-odd orientation of Figure \ref{fig:ofd0}.  

Let $K_{m,n}(z,w)$ be the corresponding modified
Kasteleyn matrix. Assume the cycle $\gamma_x$ (\resp, $\gamma_y$) 
crosses $2n$ (\resp, $2m$) edges,
whose weights are multiplied by $z$ or $z^{-1}$ (\resp, $w$ or $w^{-1}$), depending
 on their orientations. 
 Note that $K_{1,1}=K_{1}$.
 
The toroidal graph $\HH_{1,n,\Delta}$ is a line of $n$ copies of the graph 
of Figure \ref{fig:ofd0},
aligned parallel to $\g_x$. It contains $2n$  (bisector) edges with weight $c$, of
which we select two, denoted $e_1$, $e_2$, lying in the same fundamental domain. 
Let $\HH^*_{1,n,\Delta}$ be the oriented graph obtained from $\HH_{1,n,\Delta}$ by 
reversing the orientations of $e_1$ and $e_2$, 
and let $K^*_{1,n}(z,w)$ be the modified Kasteleyn matrix of $\HH^*_{1,n,\Delta}$.
 
Let $X(\la)$ be the $4n\times 4n$ matrix
\begin{equation*}
X(\la)=\begin{pmatrix}
 Y_1(\la)&0&0\\0&Y_1(\la)&0\\0&0&0
\end{pmatrix}.
\end{equation*}
Since $e_1$ and $e_2$ have weight $c$, 
\begin{align}\label{eq:ratios}
\frac{\det K^*_{1,n}(z,w)}{\det K_{1,n}(z,w)}&=
\frac{\det[K^*_{1,n}(z,w)-K_{1,n}(z,w)+K_{1,n}(z,w)]}{\det K_{1,n}(z,w)}\\
&=\frac{\det\left[X(-2c)+K_{1,n}(z,w)\right]}{\det K_{1,n}(z,w)}\nonumber\\
&=\det\left[X(-2c)K_{1,n}^{-1}(z,w)+I\right]\nonumber\\
&=\det\left[Y_2(-2c)K_{1,n}^{-1}(z,w)_{(1:4)}+I\right]\nonumber\\
&= \det\left[2cK_{1,n}^{-1}(z,w)_{(1:4)}+Y_2\left(1\right)\right],\nonumber
\end{align}
for $w=\pm 1$, since $Y_2(1)Y_2(-1)=I$ and $\det Y_2(1)=1$. 
Here, $K_{1,n}^{-1}(z,w)_{(1:4)}$ is the submatrix of $K_{1,n}^{-1}(z,w)$ comprising 
the rows and columns indexed by the four 
vertices incident with the $e_i$, see Figure \ref{fig:ofd0}.

By an explicit diagonalization of $K_{1,n}^{-1}$ as in \cite[Sect.\ 7]{ckp00},
for any two vertices $u$, $v$ in the same fundamental domain, the limit
\begin{equation}\label{lmk}
\lim_{n\to\oo} [K_{1,n}^{-1}(z,w)]_{u,v}=
\frac{1}{2\pi}\int_{0}^{2\pi}[K_{1}^{-1}(z,e^{i\theta})]_{u,v}\,d\theta
\end{equation}
exists and is independent of the choice of $w=\pm 1$.  
The proof of the next lemma is deferred until the current proof is completed.

\begin{lemma}\label{lem:det}
For  $z\in\CC$, the limit
\begin{equation}\label{eq:458}
\de(z,w)=\lim_{n\to\oo}\frac{\det K^*_{1,n}(z,w)}{\det K_{1,n}(z,w)}
\end{equation}
satisfies $\de(z,-1)=\de(z,1) =\pm 1$.
\end{lemma}

We deduce that
\begin{alignat*}{2}
\det\psi(z)&=\lim_{n\to\oo}\det\left[Y_2\left(1\right)+2cK_{1,n}^{-1}(z,1)_{(1:4)}\right]
\q&&\text{by \eqref{psz}, \eqref{pszi}, \eqref{lmk}}\\
&=\de(z,1)=\pm 1&&\text{by \eqref{eq:ratios} and Lemma \ref{lem:det}.}
\end{alignat*}
Setting $z=1$, we have by \eqref{eq:458} that
$\det \psi(1)\ge 0$ since it is the limit of a ratio
of determinants of two anti-symmetric matrices.
By \eqref{psz} and the forthcoming \eqref{k1i}, $\psi$ is continuous on the unit circle when 
the spectral curve does not intersect $\TT^2$, and
the claim follows.
(Note that \eqref{k1i} is a general fact whose proof does not depend on Lemma \ref{lem:lem102}.)
Therefore, $\psi(z)=1$ for $|z|=1$.
\end{proof}

\begin{proof}[Proof of Lemma \ref{lem:det}]
By \eqref{eq:ratios}--\eqref{lmk}, $\de(z,-1)=\de(z,1)=:\de(z)$, say.
We claim that
\begin{equation}\label{eq:457}
\det K^*_{1,n}(z,-1)=\det K_{1,n}(z,1),\qq\det K^*_{1,n}(z,1)=\det K_{1,n}(z,-1).
\end{equation}
By \eqref{eq:458}--\eqref{eq:457}, $\de(z,-1)=1/\de(z,1)$, so that $\de(z)=\pm 1$ as claimed.

We prove \eqref{eq:457} next.
Each non-vanishing term in the expansion of $\det K_{1,n}^*(z,-1)$ 
and $\det K_{1,n}(z,1)$  
corresponds to a cycle configuration on 
$\HH_{1,n,\Delta}^*$, that is, a configuration of cycles and doubled edges
in which each vertex has two incident edges. 
Let $C$ be an oriented cycle of $\HH_{1,n,\De}^*$ 
viewed as an unoriented graph, and recall that $e_1$ and $e_2$ have opposite
orientations in $\HH_{1,n,\Delta}$ and $\HH_{1,n,\Delta}^*$.
It suffices that $C$ contributes the same sign on both sides
of the left equation of \eqref{eq:457}. Let $c(C)$ (\resp, $w(C)$) be the number of
$c$-type (\resp, $w$-type) edges crossed by $C$.
By a consideration of parity, $c(C)$ is even if and
only if $w(C)$ is even, and in this case $C$ contributes the same sign. 
We claim that $w(C)=1$ if $c(C)=1$. It is standard that $C$ is either contractible or essential, 
and that $C$, if essential, has homology type $\pm 1$ in the direction $\g_x$.
Therefore, $w(C)=1$, and the claim follows.
The second equation of \eqref{eq:457} follows similarly.
\end{proof}

We remind the reader of Widom's theorem.

\begin{theorem}[Widom \cite{HW0,HW}] \label{wi}
Let $T_m(\xi)$ be a finite
block Toeplitz matrix with given symbol $\xi$ and $m\times m$ blocks.  Assume
\begin{gather*}
\sum_{k=-\oo}^{\oo}\|\xi_k\|+\left(\sum_{k=-\oo}^{\oo}|k|\cdot
\|\xi_k\|^2\right)^{\frac{1}{2}}<\oo,\\
\det\xi(e^{i\theta})\neq 0,\qquad \frac{1}{2\pi}\Delta_{0\leq\theta\leq
2\pi}\arg\det\xi(e^{i\theta})=0,
\end{gather*}
where $\|\cdot\|$ denotes Hilbert--Schmidt norm, $\xi_k$ is the $k$th Fourier coefficient of $\xi$, and  
\begin{equation*}
\frac{1}{2\pi}\Delta_{0\leq\theta\leq
2\pi}\arg\det\xi(e^{i\theta})=
\frac{1}{2\pi i}\int_{|\zeta|=1} d\zeta\, \frac{\partial}{\partial\zeta}\log\det\xi(\zeta).
\end{equation*}
Then
\begin{equation*}
\lim_{m\to\oo}\frac{\det T_m(\xi)}{G(\xi)^{m+1}}=E(\xi),
\end{equation*}
where
\begin{align}\label{eq:wi}
G(\xi)&=\exp\left\{\frac{1}{2\pi}\int_{0}^{2\pi}\log\det\xi(e^{i\theta})\,d\theta\right\},\\
E(\xi)&=\det \left[T(\xi)T(\xi^{-1})\right],\label{eq:wi2}
\end{align}
where $T(\xi)$ is the semi-infinite Toeplitz matrix with symbol $\xi$,
and the last $\det$ refers to the determinant defined for
operators on Hilbert space differing from the identity by an
operator of trace class.
\end{theorem}

We note that, when the spectral curve does not intersect the unit torus,
 $\psi$ given by \eqref{psz} is a smooth matrix-valued function on the unit circle, whence 
\begin{equation*}
\|\psi_k\|\leq C\alpha^k,
\end{equation*}
for some $C>0$, $\alpha\in(0,1)$, and all $k > 0$,
where $\psi_k$ is the $k$th Fourier coefficient of $\psi$.

\begin{proof}[Proof of Theorem \ref{anly}]
This holds  as in the proofs of \cite[Lemmas 4.4--4.7]{ZL1}, and the full
details are omitted.  Here is an outline.

The symbol $\psi$ is a $4\times 4$ matrix-valued function.
 Let $n$ be the number of NW edges in the path of \eqref{eq:condition0}
connecting $e$ and $f$. 
By the computations of Section \ref{sec:morepf} (see \eqref{eq:236} and \eqref{psz}), 
\begin{equation}\label{eq:345}
\langle \si_e\si_f \rangle^2=\det T_n(\psi),
\end{equation}
where $T_n(\psi)$ is a truncated block Toeplitz matrix consisting of the first 
$n\times n$ blocks of an infinite block Toeplitz matrix $T(\psi)$ with symbol $\psi$
(so that $T_n(\psi)$ is a $4n \times 4n$ matrix). 
By Lemma \ref{lem:lem102}, when the spectral
curve does not intersect the unit torus $\TT^2$, we have $G(\psi)=1$, 
where $G(\psi)$ is given in \eqref{eq:wi}. By Theorem \ref{wi},
\begin{equation}\label{eq:lambdalim}
\La(a,b,c)=\lim_{n\to\oo}\det T_n(\psi)=E(\psi),
\end{equation} 
where $E(\psi)$ is given in \eqref{eq:wi2}.

Non-analyticity of $\La$ may arise only as follows. One may write
\begin{equation}
 [K_1^{-1}(z,w)]_{i,j}=\frac{Q_{i,j}(z,w)}{P(z,w)},\label{k1i}
\end{equation} 
where $Q_{i,j}(z,w)$ is a Laurent polynomial in $z$, $w$ derived in terms of certain cofactors of 
$K_1(z,w)$, and $P(z,w)=\det K_1(z,w)$ is the characteristic polynomial of the dimer model
on $\HnD$. It follows that $\La$ is analytic when $P(z,w)$ has no
zeros on the unit torus $\TT^2$.
The last occurs only under the condition of Proposition \ref{prop:Pzero},
and the claim follows.
\end{proof}

\section{Proof of exponential convergence in Theorem \ref{thm:main}($\mathrm{c}$)}\label{sec:pf32}

We develop the method of proof of Widom's formula, Theorem \ref{wi},
see \cite{HW0,HW}. 
For a positive integer  $r$, let $A_r\cap K_r$ be the Banach algebra of $r\times r$ 
matrix-valued functions on the unit circle under the norm
\begin{equation*}
\|\phi\|=\sum_{k=-\oo}^{\oo}\|\phi_k\|+
\left(\sum_{k=-\oo}^{\oo}|k|\cdot\|\phi_k\|^2\right)^{\frac{1}{2}},
\end{equation*}
where $\phi_k$ is the $k$th 
Fourier coefficient of the matrix-valued function $\phi$. Note that the trigonometric polynomials are dense
in $A_r\cap K_r$. 

The main theorem of this section is as follows.

\begin{theorem}\label{ms8}
Let $T(\psi)$ be a semi-infinite 
block Toeplitz matrix with symbol $\psi$, where $\psi$ is an $r\times r$ matrix-valued, 
$C^{\infty}$-function on the unit circle. Let $T_n(\psi)$ be the truncated block Toeplitz matrix consisting of the first $n\times n$
blocks of $T(\psi)$. In the limit as $n\to\infty$, $\det T_n(\psi)$ converges to its limit exponentially fast.
\end{theorem}

We recall \eqref{eq:345} from the last section, and note that
$\psi$ is a matrix-valued function in $A_4\cap K_4$. 
Let $H(\psi)$ be the Hankel matrix with symbol $\psi$, 
\begin{equation*}
H(\psi)=\bigl(\psi_{i+j+1}\bigr)_{0\leq i,j<\oo},
\end{equation*}
and write 
\begin{equation}\label{eq:346}
\wt\psi(z)=\psi(z^{-1}).
\end{equation}

\begin{lemma}
Let $e$, $f$ be NW edges of the hexagonal lattice satisfying \eqref{eq:condition0}.
In the representation \eqref{eq:345} of the edge--edge correlation 
$\langle \si_e\si_f\rangle_n$,  the symbol $\psi$ is $C^{\infty}$ on the unit circle
whenever the spectral 
curve does not intersect the unit torus.
\end{lemma}

\begin{proof}
We recall from Proposition \ref{prop:Pzero} that, when $a\geq b,c >0$  and
$$
\sqrt{a}\neq \sqrt{b}+\sqrt{c}, \qq \sqrt{c}\neq \sqrt{a}+\sqrt{b},
$$
the characteristic polynomial $P(z,w)$ 
has no zeros on the unit torus. As in Section \ref{sec:morepf}, 
$\psi(\xi)$ is a matrix-valued function defined on the unit circle, each entry of which has the form
\begin{equation*}
\frac{1}{2\pi}\int_{0}^{2\pi}\frac{Q(\xi,e^{i\phi})}{P(\xi,e^{i\phi})}\,d\phi,
\end{equation*}
where $P(z,w)$ is the characteristic polynomial, and $Q(z,w)$ is a Laurent polynomial. 
When $P(z,w)$ has no zeros on the unit torus, $\psi(\xi)$ is a $C^{\oo}$ function on the unit circle, 
and the $n$th Fourier coefficient of $\psi(\xi)$ decays exponentially  to 0 as $|n|\to\oo$. 
\end{proof}

\medskip

\par\noindent
{\bf Assume first that the operator $T(\wt{\psi})$ is invertible as an operator 
on $\ell_2$ sequences of $r$-vectors}.
This is equivalent to 
assuming that the matrix-valued function $\psi$ has a factorization of the form
\begin{equation}
\psi=\psi_{+}\psi_{-},\label{ppm}
\end{equation}
where the $\psi_{\pm}$ are invertible in $A_r\cap K_r$, and the $\psi_{+}^{\pm1}$ 
(\resp, $\psi_{-}^{\pm 1}$) have Fourier coefficients that 
vanish for negative (\resp, positive) indices. As in \cite[Sect.\ 3]{HW}, we have
\begin{equation}
\frac{\det T_n(\psi)}{G(\psi)^{n+1}}=
\det\Bigl(I-P_nH(\psi)H(\wt{\psi}_{-}^{-1})P_nT(\psi_{+}^{-1})P_n\Bigr),\label{dq}
\end{equation}
where $P_nA$ is the submatrix of $A$ consisting of its first $nr$ rows, 
and $BP_n$ is the submatrix of $B$ consisting of its first $nr$ columns. Recall that $G$ is given by \eqref{eq:wi}.

Now we define operator norms, and discuss inequalities regarding these norms.

\begin{definition}
For a compact operator $A$ on a Hilbert space, and $1\leq p\leq \oo$, 
let $\|A\|_p$ denote the $p$-norm of the eigenvalue-sequence 
of $(A^*A)^{\frac{1}{2}}$, 
where $A^*$ is the conjugate transpose of $A$. 
The $\infty$-norm is the usual operator norm and is so defined even if $A$ is not compact. The
$2$-norm is the Hilbert--Schmidt norm, and the $1$-norm is the trace-norm.
The set of compact operators with finite $p$-norm is denoted by $\varPhi_p$. 
Thus $\varPhi_1$ is the set
of  operators of trace class; 
$\varPhi_2$ is the set of Hilbert--Schmidt operators; 
and $\varPhi_{\infty}$ is the set of compact operators.
\end{definition}

As in  \cite[Sect.\ 2]{HW}, we have the following lemma
\begin{lemma}\mbox{}
\begin{letlist}
\item Let $1\leq p\leq \infty$. If $A\in\varPhi_{p}$ and $B,C\in\varPhi_{\infty}$, then $BAC\in\varPhi_p$. Moreover, 
\begin{align*}
\|BAC\|_p\leq \|A\|_p\|B\|_{\oo}\|C\|_{\oo}.
\end{align*}
\item If $A,B\in\varPhi_2$, then $AB\in \varPhi_1$. Moreover,
\begin{align*}
\|AB\|_1\leq \|A\|_2\|B\|_2.
\end{align*}
\end{letlist}
\end{lemma}

\begin{lemma}
Assume that $\psi(\zeta)$ is a $C^{\infty}$ $r\times r$ matrix-valued function on the unit circle
with exponential decaying Fourier coefficients. 
Assume also that $\psi$ has a factorization given by \eqref{ppm}.
 Then $\psi_+^{\pm 1}$, $\psi_-^{\pm1}$, $\wt{\psi}_{+}^{\pm1}$,
$\wt{\psi}_-^{\pm1}$ are all $C^{\infty}$ $r\times r$ matrix-valued functions on the unit circle with exponential decaying 
Fourier coefficients.
\end{lemma}

\begin{proof}
By the arguments of \cite[p.\ 10]{HW}, when $T(\wt{\psi})$ is invertible,
$\wt{\psi}_{-}^{-1}$ is a matrix-valued function whose sequence of non-negative Fourier 
coefficients is $T^{-1}(\wt{\psi})(I,0,0,\dots)$. By  \cite[Thm 1.3]{TS02},   
the entries in the first column of $T(\wt{\psi})^{-1}$ decay exponentially as the 
entry moves away from the diagonal, whence the $n$th Fourier coefficient of $\wt{\psi}_{-}^{-1}$ 
decays to zero exponentially as $|n|\to\oo$.

The exponential decay of Fourier coefficients and smoothness of $\psi_+^{\pm 1}$, $\psi_-^{\pm1}$, $\wt{\psi}_{+}^{\pm1}$,
$\wt{\psi}_-^{\pm1}$ follow.
\end{proof}

By explicit computations (see \cite{HW}), 
\begin{equation*}
H(\psi)H(\wt{\psi}_{-}^{-1})T(\psi_{+}^{-1})=H(\psi)H(\wt{\psi}^{-1}).
\end{equation*}
Moreover,
\begin{align*}
\Bigl\|H(\psi)H(\wt{\psi}_{-}^{-1})P_n&T(\psi_{+}^{-1})-H(\psi)H(\wt{\psi}_{-}^{-1})T(\psi_{+}^{-1})\Bigr\|_1\\
&=\bigl\|H(\psi)H(\wt{\psi}_{-}^{-1})(P_n-I)T(\psi_{+}^{-1})\bigr\|_1\\
&\leq\bigl\|H(\psi)H(\wt{\psi}_{-}^{-1})(P_n-I)\bigr\|_1\cdot \|T(\psi_{+}^{-1})\|_{\oo}\\
&\leq\|H(\psi)\|_2\cdot\|H(\wt{\psi}_{-}^{-1})(P_n-I)\|_2\cdot\|T(\psi_{+}^{-1})\|_{\oo}.
\end{align*}

\begin{lemma}
Assume that $\psi$ is a $r\times r$ matrix-valued, $C^{\infty}$ function on the unit circle
with exponential decaying Fourier coefficients, and assume that $\psi$ has a factorization given by (\ref{ppm}).
Let $\tr(A)$ be the trace of $A$. Then there exists $0<\beta_1<1$ such that
\begin{align}
\|T(\psi_+^{-1})\|_{\oo}&<\oo,\label{in1}\\
\|H(\psi)\|_2&=\bigl[\tr(H^*(\psi)H(\psi))\bigr]^{\frac{1}{2}}<\oo,\label{in2}\\
\|H(\wt\psi_{-}^{-1})(P_n-I)\|_2&=
\tr\bigl[(P_n-I)H^*(\wt{\psi}_{-}^{-1})H(\wt{\psi}_{-}^{-1})(P_n-I)\bigr]^{\frac{1}{2}}<\beta_1^n.\label{in3}
\end{align}
\end{lemma}

\begin{proof}
Note that $\|T(\psi_+^{-1})\|_{\oo}$ is exactly the maximal singular 
value of $T(\psi_+^{-1})$. Note also that $T(\psi_+)^{-1}$ 
is an 
upper triangular matrix whose off-diagonal entries decay exponentially fast. 
Let
\begin{equation*}
A=[T(\psi_+)^{-1}]^*T(\psi_+)^{-1}=(a_{i,j})_{i,j\in \NN}.
\end{equation*}
It is not hard to check that the off-diagonal entries of $A$ decay exponentially fast 
with respect to their distance to the diagonal. Moreover, 
$\|T(\psi_+^{-1})\|_{\oo}^2$ is the maximal eigenvalue of $A$. 
It is well known that
\begin{equation*}
\|T(\psi_+^{-1})\|_{\oo}^2\leq\max_{j} \sum_{i=1}^{\infty}|a_{ij}|<\infty,
\end{equation*}
and \eqref{in1} follows.
The inequality \eqref{in2} is obtained since the diagonal entries of 
$H^*(\psi)H(\psi)$ decay exponentially when moving down the diagonal. 
Equation  \eqref{in3} is similar.
\end{proof}

Now,
\begin{align}\label{exd1}
\Bigl\|P_nH(\psi)H(&\wt{\psi}^{-1})P_n-H(\psi)H(\wt{\psi}^{-1})\Bigr\|_1\\
&\leq \bigl\|P_nH(\psi)H(\wt{\psi}^{-1})(P_n-I)\bigr\|_1+\bigl\|(P_n-I)H(\psi)H(\wt{\psi}^{-1})\bigr\|_1\notag\\
&\leq \|P_n\|_{\oo}\cdot\|H(\psi)H(\wt{\psi}^{-1})(P_n-I)\|_1+\|(P_n-I)H(\psi)\|_2\cdot\|H(\wt{\psi}^{-1})\|_2\notag\\
&\leq\|H(\psi)\|_2\cdot\|H(\wt{\psi}^{-1})(P_n-I)\|_2+\|(P_n-I)H(\psi)\|_2\cdot\|H(\wt{\psi}^{-1})\|_2\notag\\
&\leq \beta_2^n,\notag
\end{align}
where $0<\beta_2<1$. 
The last inequality  holds because the $(i,j)$ entries of 
$H(\psi)$ and $H(\wt{\psi})$ decay exponentially in $i+j$.

The following cases may occur:
\begin{letlist}
\item[I.] $I-H(\psi)H(\wt{\psi}^{-1})$ is invertible,
\item[II.] $I-H(\psi)H(\wt{\psi}^{-1})$ is not invertible.
\end{letlist}

\par\noindent
\emph{Assume Case I occurs.}
By the formula of \cite[p.\ 116]{ggk},
\begin{align}
\Bigl|\det\bigl(I-P_nH(\psi)&H(\wt{\psi}_{-}^{-1})P_nT(\psi_{+}^{-1})P_n\bigr)-
\det\bigl(I-H(\psi)H(\wt{\psi}^{-1})\bigr)\Bigr|\label{es2}\\
&\leq e^{\|H(\psi)H(\wt{\psi}^{-1})\|_1}
(e^{Q_n}-1)
\notag\\
&\leq \beta_3^n,\notag
\end{align}
where 
$$
Q_n:=\left\|\bigl(I-H(\psi)H(\wt{\psi}^{-1})\bigr)^{-1}\Bigl(P_nH(\psi)H(\wt{\psi}_{-}^{-1})P_nT(\wt{\psi}_{+}^{-1})P_n-H(\psi)H(\wt{\psi}^{-1})\Bigr)\right\|_1
$$
and $0<\beta_3<1$.

\par\noindent
\emph{Assume  Case II occurs.}
We follow the approach of \cite[pp.\ 116--117]{ggk}. 
When $I-H(\psi)H(\wt{\psi}^{-1})$ is not invertible, the point 1 is an isolated 
eigenvalue of finite type for $H(\psi)H(\wt{\psi}^{-1})$. Let $P$ be the 
corresponding Riesz projection, and put $H_1=\mathrm{Im}\,P$, $H_2=\mathrm{Ker}\,P$, 
so that $H_1$ is finite-dimensional. For simplicity, we write
\begin{align*}
F_n&:=P_nH(\psi)H(\wt{\psi}_{-}^{-1})P_nT(\psi_{+}^{-1})P_n,\\
A&:=H(\psi)H(\wt{\psi}^{-1}).
\end{align*}
With respect to the decomposition $H=H_1\oplus H_2$, 
we have
\begin{equation*}
F_n=\begin{pmatrix} 
K_{11}^{(n)} &K_{12}^{(n)}\\
K_{21}^{(n)}&K_{22}^{(n)}
\end{pmatrix},\qq 
A=\begin{pmatrix}
A_{11}&0\\
0&A_{22}
\end{pmatrix},
\end{equation*}
and it follows  as in \eqref{exd1} that
\begin{equation*}
\|F_n-A\|_2\leq \|F_n-A\|_1\leq \beta_4^n,
\end{equation*}
for some $0<\beta_4<1$. Moreover,
\begin{equation*}
(F_n-A)^*(F_n-A) =
\begin{pmatrix}
L_1 & \cdots\\
\cdots & L_2
\end{pmatrix},
\end{equation*}
where
\begin{align*}
L_1 &= (K_{11}^{(n)}-A_{11})^*(K_{11}^{(n)}-A_{11})+(K_{21}^{(n)})^*K_{21}^{(n)},\\
L_2 &=(K_{12}^{(n)})^*K_{12}^{(n)}+(K_{22}^{(n)}-A_{22})^*(K_{22}^{(n)}-A_{22}).
\end{align*}
Therefore,
\begin{align*}
\|F_n-A\|_2^2&=\tr\left((K_{11}^{(n)}-A_{11})^*(K_{11}^{(n)}-A_{11})\right)+\tr\left((K_{21}^{(n)})^*K_{21}^{(n)}\right)\\
&\hskip1cm +\tr\left((K_{22}^{(n)}-A_{22})^*(K_{22}^{(n)}-A_{22})\right)+\tr\left((K_{12}^{(n)})^*K_{12}^{(n)}\right)\\
&=\|K_{11}^{(n)}-A_{11}\|_2^2+\|K_{21}^{(n)}\|_2^2+\|K_{12}^{(n)}\|_2^2+\|K_{22}^{(n)}-A_{22}\|_2^2.
\end{align*}
Hence,
\begin{align*}
\|K_{11}^{(n)}-A_{11}\|_2&\leq \beta_4^n,\\
\|K_{12}^{(n)}\|_2&\leq \beta_4^n,\\
\|K_{21}^{(n)}\|_2&\leq \beta_4^n.
\end{align*}

As in \cite[pp.\ 116--117]{ggk}, with each $I_j$ an identity matrix of suitable size,
\begin{equation*}
\det(I-F_n)=\det\left(I_1-K_{11}^{(n)}+K_{12}^{(n)}(I_2-K_{22}^{(n)})^{-1}K_{21}^{(n)}\right)\det \bigl(I_2-K_{22}^{(n)}\bigr),
\end{equation*}
and
\begin{equation*}
\lim_{n\to\oo}\det\bigl (I_2-K_{22}^{(n)}\bigr)=\det(I-A_{22}).
\end{equation*}
Moreover, 
\begin{align*}
\Bigl\|I_1-K_{11}^{(n)}+{}&K_{12}^{(n)}(I_2-K_{22}^{(n)})^{-1}K_{21}^{(n)}-(I_1-A_{11})\Bigr\|_1\\
&\leq \|K_{11}^{(n)}-A_{11}\|_1+\|K_{12}^{(n)}\|_2\cdot\|I_2-K_{22}^{(n)}\|_{\oo}
\cdot\|K_{21}^{(n)}\|_2\\
&\leq \beta_6^n,
\end{align*}
for some $0<\beta_6<1$. Since $I_1-K_{11}^{(n)}+K_{12}^{(n)}(I_2-K_{22}^{(n)})^{-1}K_{21}^{(n)}$
 is a finite matrix, all the norms are equivalent,  we have  
\begin{equation*}
\left|\det\left(I_1-K_{11}^{(n)}+K_{12}^{(n)}(I_2-K_{22}^{(n)})^{-1}K_{21}^{(n)}\right)-\det(I_1-A_{11})\right|\leq \beta_7^n,
\end{equation*}
for some $0< \beta_7<1$.

Since (by Lemma \ref{lem:lem102}) $G(\psi)=1$, we obtain that, when $T(\wt{\psi})$ is invertible, $\langle \si_e\si_f \rangle^2$ converges to its limit exponentially fast in the limit
as  $|e-f|\to\oo$. This completes the consideration of Case II.
\medskip

\par\noindent
{\bf Assume now that $T(\wt{\psi})$ is not invertible.} 
Since $T(\wt{\psi})$ is Fredholm of index $0$, by \cite{HW}, 
there exists $\phi$ with only finitely many non-zero Fourier coefficients, 
such that $T(\wt{\psi}+\eps\wt{\phi})$ is invertible for  
sufficiently small $\eps\neq 0$. 
Therefore,
\begin{equation*}
\lim_{n\to\oo}\frac{\det T_n(\psi+\eps\phi)}{G(\psi+\eps\phi)^{n+1}}=\det \bigl[T(\psi+\eps\phi)T\bigl((\psi+\eps\phi)^{-1}\bigr)\bigr],
\end{equation*}
for $\eps$ belonging to some punctured disk with centre $0$. 
Using the same arguments as above we obtain that, for 
sufficiently small  $|\eps|>0$,
\begin{equation}
\left|\frac{\det T_n(\psi+\eps\phi)}{G(\psi+\eps\phi)^{n+1}}-
\det \bigl[T(\psi+\eps\phi)T\bigl((\psi+\eps\phi)^{-1}\bigr)\bigr]
\right|\leq \beta_5^n\label{tg}
\end{equation}
for some $0<\beta_5<1$ independent of $\eps$. 
This is obtained by first
expressing the ratio on the left side of \eqref{tg} as in 
\eqref{dq}, then
follow the previous argument. Note that $\beta_5$ depends on the
exponential decay rate of $(\psi+\eps\phi)_k$ as $k\rightarrow\infty$, which can be made
universal for $\eps$ with $|\eps|$ sufficiently small.
 Since the component
functions on the left side of \eqref{tg} are analytic in $\eps$ on some disk $\{\epsilon\in\CC:|\epsilon|\leq c_0\}$, 
by the maximum principle, \eqref{tg} holds for $\eps=0$ also. 
Hence $\langle\si_e\si_f\rangle^2$ converges to its limit exponentially fast. 
The proof of Theorem \ref{ms8}  is complete. 

\section{Alternative proof of Theorem \ref{thm:main}($\mathrm{b}$)}
\label{sec:eecad}

We find it slightly more convenient to work here with horizontal edges
rather than NW edges, and there is no essential difference in the proof. 
Let $e, f\in\EE$ be horizontal edges.  There exists a horizontal edge $g \in \EE$ such that:
$e$ and $g$ 
(\resp, $f$ and $g$) are connected by a path $\ell_{eg}$ (\resp, $\ell_{fg}$) of $A\HH$ comprising 
only horizontal and NW (\resp, NE) edges, and the unique common edge of
$\ell_{eg}$ and $\ell_{fg}$ is $g$.  Let $m+1$ (\resp, $n+1$) be the number
of horizontal edges in $\ell_{eg}$ (\resp, $\ell_{fg}$). Write $m \wedge n=\min\{m,n\}$
and $m \vee n = \max\{m,n\}$.

\begin{theorem}\label{thm:main4}
Assume the parameters of the 1-2 model satisfy
\begin{equation}\label{eq:ass5}
 a\geq b>0,  \q c>0,\q \sqrt{a}\neq \sqrt{b}+\sqrt{c},\q  \sqrt{c}\neq\sqrt{a}+\sqrt{b}.
 \end{equation}
\begin{letlist}
\item The limit 
$$
\La(a,b,c):= \lim_{m,n\to\oo}\langle\si_e\si_f\rangle^2
$$
exists, and
\begin{equation}
\left|\langle\si_e\si_f \rangle^2-\La\right|
\leq C\alpha^{m \wedge n},\label{dr}
\end{equation}
where $C>0$, $\alpha\in(0,1)$ are constants independent of $m$, $n$. 

\item If, in addition, $\sqrt{a}<\sqrt{b}+\sqrt{c}$ and $\sqrt{c}<\sqrt{a}+\sqrt{b}$,
then $C>0$, $\al\in(0,1)$ may be chosen such that
\begin{equation}
|\langle \si_e\si_f \rangle|\leq C\alpha^{|e-f|}.\label{eds}
\end{equation}
\end{letlist}
\end{theorem}

\begin{proof}
(a)
By a computation similar to that leading to \eqref{eq:236}, $\langle\si_e\si_f\rangle^2$
may be expressed in the form
\begin{align*}
\langle\si_e\si_f\rangle^2=\det\begin{pmatrix}
T_m(\psi_1)& A_{m,n}\\B_{n,m}&T_{n}(\psi_2)\end{pmatrix},
\end{align*}
where $\psi_i$ is a $4 \times 4$ matrix-valued function (see \eqref{psz}),
$T_r(\psi_i)$ is a $4r\times 4r$ truncated block Toeplitz matrix with symbol $\psi_i$,
and  $A_{m,n}$ (\resp, $B_{n, m}$) is a $4m\times 4n$ (\resp, $4n\times 4m$) matrix with entries 
satisfying
\begin{align*}
|A_{m,n}(i,j)|\leq C\beta_1^{i+j},\qq
|B_{n,m}(i,j)|\leq C\beta_1^{i+j},
\end{align*} 
where $C>0$, $\beta_1\in(0,1)$ are constants independent of $i$, $j$.  

Let $\phi_1$, $\phi_2$ be two matrix-valued functions defined on the unit circle with only finitely many non-vanishing Fourier coefficients, such that $T(\wt{\psi}_1+\eps\wt{\phi}_1)$ and 
$T(\wt{\psi}_2+\eps\wt{\phi}_2)$ are invertible for  complex  $\eps\neq 0$
with sufficiently small modulus (recall \eqref{eq:346}).
Let $\eps>0$ be given accordingly, and let
\begin{align*}
\psi_{i,\eps}=\varphi_i+\eps\phi_i,\qq i=1,2.
\end{align*}
Since $T(\wt{\psi}_{i,\eps})$ is invertible, we have (as in Section \ref{sec:pf32}) that
\begin{align*}
\psi_{i,\eps}=\psi_{i,\eps, +}\psi_{i,\eps,-},\qq i=1,2,
\end{align*}
where the $\psi_{i,\eps,\pm}$ are invertible in $A_4\cap K_4$, and $\psi_{i,\eps,+}^{\pm 1}$ (\resp, $\psi_{i,\eps,-}^{\pm 1}$) have Fourier coefficients that vanish for negative (\resp, positive) indices. 

Subject to \eqref{eq:ass5}, by Proposition \ref{prop:Pzero} 
the spectral curve has no zeros on the unit torus. 
As explained after Theorem \ref{wi}, the entries of the Toeplitz matrix (Fourier coefficients of a smooth 
function on the unit circle) decay exponentially as their distances from the
diagonal go to infinity.
 
By computations similar to those of \cite[p.\ 8]{HW},
\begin{align}
&\begin{pmatrix}T_m(\psi_{1,\eps})& A_{m,n}\\ B_{n,m}& T_n(\psi_{2,\eps})\end{pmatrix}
\begin{pmatrix}
T_{m}(\psi_{1,\eps,-}^{-1})T_{m}(\psi_{1,\eps,+}^{-1})&0\\0& T_n(\psi_{2,\eps,-}^{-1})T_n(\psi_{2,\eps,+}^{-1})\end{pmatrix}\label{dt}\\
&\hskip1cm =\begin{pmatrix}
T_m(\psi_{1,\eps})T_{m}(\psi_{1,\eps,-}^{-1})T_{m}(\psi_{1,\eps,+}^{-1})&A_{m,n}T_{n}(\psi_{2,\eps,-}^{-1})T_{n}(\psi_{2,\eps,+}^{-1})\\B_{n,m}T_{m}(\psi_{1,\eps,-}^{-1})T_{m}(\psi_{1,\eps,+}^{-1})&T_n(\psi_{2,\eps})T_{n}(\psi_{2,\eps,-}^{-1})T_{n}(\psi_{2,\eps,+}^{-1})\end{pmatrix}\notag\\
&\hskip1cm =I_{4m+4n} - S_{m,n},\notag
\end{align}
where $I_r$ is the $r\times r$ identity matrix 
and
\begin{equation*}
S_{m,n}=\begin{pmatrix}
P_mH(\psi_{1,\eps})H(\wt{\psi}_{1,\eps,-}^{-1})P_mT_m(\psi_{1,\eps,+}^{-1})&-A_{m,n}T_{n}(\psi_{2,\eps,-}^{-1})T_{n}(\psi_{2,\eps,+}^{-1})\\-B_{n,m}T_{m}(\psi_{1,\eps,-}^{-1})T_{m}(\psi_{1,\eps,+}^{-1})&P_nH(\psi_{2,\eps})H(\wt{\psi}_{2,\eps,-}^{-1})P_nT_n(\psi_{2,\eps,-}^{-1})\end{pmatrix}.
\end{equation*}

We now take determinants of \eqref{dt}.
As in \cite[p.\ 8]{HW}, 
\begin{align*}
\det[T_{m}(\psi_{1,\eps,-}^{-1})]\det[T_{m}(\psi_{1,\eps,+}^{-1})]&=G(\psi_{1,\epsilon,-}^{-1})^{m+1}G(\psi_{1,\epsilon,+}^{-1})^{m+1}\\
&=\frac{1}{G(\psi_{1,\epsilon})^{m+1}},
\end{align*}
whence
\begin{align}
\frac{1}{G(\psi_{1,\eps})^{m+1}G(\psi_{2,\eps})^{n+1}}
\det\begin{pmatrix}
T_m(\psi_{1,\eps})&A_{m,n}\\B_{n,m}&T_{n}(\psi_{2,\eps})\end{pmatrix}
=\det\left(I_{4m+4n}-S_{m,n}\right).\label{dtt}
\end{align}

Let
\begin{align}
S=\begin{pmatrix}
H(\psi_{1,\eps})H(\wt{\psi}_{1,\eps,-}^{-1})T(\psi_{1,\eps,+}^{-1})
&-AT(\psi_{2,\eps,-}^{-1})T(\psi_{2,\eps,+}^{-1})\\
-BT(\psi_{1,\eps,-}^{-1})T(\psi_{1,\eps,+}^{-1})
&H(\psi_{2,\eps})H(\wt{\psi}_{2,\eps,-}^{-1})T(\psi_{2,\eps,-}^{-1})\end{pmatrix},
\label{ds}
\end{align}
where $A$ and $B$ are infinite matrices obtained as the limits of $A_{m,n}$ and $B_{n,m}$, as $m,n\to\oo$. Note that $S$ is a trace-class operator. 
Therefore, $\det(I-S)$ is well-defined and complex analytic in $a,b,c$, 
whenever the entries of $S$ are analytic in $a,b,c$ (see \cite{GK69}, 
and also \cite[Lemma 3.1]{ggk} and  \cite[Lemma 4.6]{ZL1}).
Then,
\begin{align*}
\|S_{m,n}-S\|_1
&\leq \bigl\|A_{m,n}T_{n}(\psi_{2,\eps,-}^{-1})T_{n}(\psi_{2,\eps,+}^{-1})-AT(\psi_{2,\eps,-}^{-1})T(\psi_{2,\eps,+}^{-1})\bigr\|_1\\
&\qq+\bigl\|B_{n,m}T_{m}(\psi_{1,\eps,-}^{-1})T_{m}(\psi_{1,\eps,+}^{-1})-BT(\psi_{1,\eps,-}^{-1})T(\psi_{1,\eps,+}^{-1})\bigr\|_1\\
&\qq+\bigl\|P_mH(\psi_{1,\eps})H(\wt{\psi}_{1,\eps,-}^{-1})P_mT_m(\psi_{1,\eps,+}^{-1})-H(\psi_{1,\eps})H(\wt{\psi}_{1,\eps,-}^{-1})T(\psi_{1,\eps,+}^{-1})\bigr\|_1\\
&\qq+\bigl\|P_nH(\psi_{2,\eps})H(\wt{\psi}_{2,\eps,-}^{-1})P_nT_n(\psi_{2,\eps,-}^{-1})-H(\psi_{2,\eps})H(\wt{\psi}_{2,\eps,-}^{-1})T(\psi_{2,\eps,-}^{-1})\bigr\|_1.
\end{align*}
Using similar arguments as in \eqref{exd1}, we can show that
\begin{align*}
\bigl\|P_mH(\psi_{1,\eps})H(\wt{\psi}_{1,\eps,-}^{-1})P_mT_m(\psi_{1,\eps,+}^{-1})-H(\psi_{1,\eps})H(\wt{\psi}_{1,\eps,-}^{-1})T(\psi_{1,\eps,+}^{-1})\bigr\|_1 &\leq \beta_2^m,\\
\bigl\|P_nH(\psi_{2,\eps})H(\wt{\psi}_{2,\eps,-}^{-1})P_nT_n(\psi_{2,\eps,-}^{-1})-H(\psi_{2,\eps})H(\wt{\psi}_{2,\eps,-}^{-1})T(\psi_{2,\eps,-}^{-1})\bigr\|_1 &\leq \beta_2^n,
\end{align*}
where $\beta_2\in(0,1)$ is a constant independent of $m,n,\eps$. Moreover, 
\begin{align*}
|A(i,j)| \leq C\beta_1^{i+j},\q
|B(i,j)| \leq C\beta_1^{i+j},\q
|T(\psi_*)(i,j)| \leq C\beta_3^{|i-j|},
\end{align*}
subject to \eqref{eq:ass5}. Here,
$\psi_{*}\in\{\psi_{2,\eps,-}^{-1},\psi_{2,\eps,+}^{-1},\psi_{1,\eps,+}^{-1},\psi_{1,\eps,-}^{-1}\}$,
and $\beta_1, \beta_3\in(0,1)$ are independent of $i,j,\eps$.  Therefore,
\begin{alignat*}{4}
\bigl|AT(\psi_{2,\eps,+}^{-1})T(\psi_{2,\eps,-}^{-1})(i,j)\bigr|&&\leq C_1\beta_4^{i+j},&&\qq
\bigl|BT(\psi_{1,\eps,+}^{-1})T(\psi_{1,\eps,-}^{-1})(i,j)\bigr|&&\leq C_1\beta_4^{i+j},\\
|AT(\psi_{2,\eps,+}^{-1})(i,j)|&&\leq C_1\beta_4^{i+j},&&\qq
|AT(\psi_{2,\eps,-}^{-1})(i,j)|&&\leq C_1\beta_4^{i+j},\\
|BT(\psi_{1,\eps,+}^{-1})(i,j)|&&\leq C_1\beta_4^{i+j},&&\qq
|BT(\psi_{1,\eps,-}^{-1})(i,j)|&&\leq C_1\beta_4^{i+j},
\end{alignat*}
where $C_1>0$, $\beta_4\in(0,1)$  are constants independent of $i,j,\eps$.

Hence,
\begin{align*}
\Bigl\|P_mAT(\psi_{2,\eps,+}^{-1})T(\psi_{2,\eps,-}^{-1})P_n-AT(\psi_{2,\eps,-}^{-1})T(\psi_{2,\eps,+}^{-1})\Bigr\|_1\leq C_2\beta_5^{m \wedge n},
\end{align*}
where $C_2>0$, $\beta_5\in(0,1)$ are constants independent of $i,j,\eps$.

Moreover,
\begin{align*}
\Bigl\|P_mA &T(\psi_{2,\eps,-}^{-1})P_n T(\psi_{2,\eps,-}^{-1})P_n-P_mA T(\psi_{2,\eps,-}^{-1})T(\psi_{2,\eps,+}^{-1})P_n\Bigr\|_1\\
&=\|P_m A T(\psi_{2,\eps,-}^{-1})(P_n-I)T(\psi_{2,\eps,+}^{-1})P_n\|_1\\
&\leq \|P_m A T(\psi_{2,\eps,-}^{-1})(P_n-I)\|_1\cdot\|T(\psi_{2,\eps,+}^{-1})P_n\|_{\oo}\\
&\leq C_3 \beta_6^{n},
\end{align*}
and
\begin{align*}
\Bigl\|P_mAP_n &T(\psi_{2,\eps,-}^{-1})P_n T(\psi_{2,\eps,-}^{-1})P_n-P_mA T(\psi_{2,\eps,-}^{-1})P_nT(\psi_{2,\eps,+}^{-1})P_n\Bigr\|_1\\
&=\|P_m A (P_n-I)T(\psi_{2,\eps,-}^{-1})P_nT(\psi_{2,\eps,+}^{-1})P_n\|_1\\
&\leq \|P_m A(P_n-I)\|_1\cdot\| T(\psi_{2,\eps,-}^{-1})P_nT(\psi_{2,\eps,+}^{-1})P_n\|_{\oo}\\
&\leq  C_3 \beta_6^{n},
\end{align*}
where  $C_3>0$, $\beta_6\in(0,1)$ are  constants independent of $m,n,\eps$.

We consider the following cases:
\begin{enumerate}
\item[I.] $I-S$ is invertible,
\item[II.] $I-S$ is not invertible.
\end{enumerate}
\emph{If Case I occurs}, following arguments similar to those of Section \ref{sec:pf32},
\begin{align}
\bigl|\det(I-S_{m,n})-\det(I-S)\bigr|\leq C_4\beta_7^{m \wedge n},\label{dttt}
\end{align}
where $C_4>0$, $\beta_7\in(0,1)$  are constants independent of $m,n,\eps$.
\emph{If Case II occurs}, by considering the Riesz projection and following similar arguments, 
\begin{align}
\bigl|\det(I-S_{m,n})-\det(I-S)\bigr|\leq C_5\beta_8^{m \wedge n},\label{ddtt}
\end{align}
where $C_5>0$, $\beta_8\in(0,1)$ are constants independent of $m,n,\eps$.

Letting $\eps\to 0$ and using analyticity (as in Section \ref{sec:pf32})
subject to \eqref{eq:ass5}, we have
\begin{equation*}
\La:=\lim_{m,n\to\oo}\langle\si_e\si_f \rangle^2=\det(I-S)\big|_{\epsilon=0},
\end{equation*}
by \eqref{dtt}, \eqref{dttt}, \eqref{ddtt} 
and the fact that $G(\psi_i)=1$ 
(the last follows by Lemma \ref{lem:lem102} and \eqref{eq:wi}). Moreover,
\eqref{dr} holds.

(b)
The above argument applies also when $m$ is fixed and $n\to\oo$. 
In this case, we replace $S$ in \eqref{ds} by
\begin{align*}
S_m&=\begin{pmatrix}
P_mH(\psi_{1,\eps})H(\wt{\psi}_{1,\eps,-}^{-1})P_mT_m(\psi_{1,\eps,+}^{-1})
&-A_{m,\oo}T(\psi_{2,\eps,-}^{-1})T(\psi_{2,\eps,+}^{-1})\\
-B_{\oo,m}T_m(\psi_{1,\eps,-}^{-1})T_m(\psi_{1,\eps,+}^{-1})
&H(\psi_{2,\eps})H(\wt{\psi}_{2,\eps,-}^{-1})T(\psi_{2,\eps,-}^{-1})
\end{pmatrix},
\end{align*}
where $A_{m,\oo}$, $B_{\oo,m}$ are matrices obtained from $A_{m,n}$, $B_{n,m}$ by letting $n\to\oo$. 
Now, $\det(I-S_m)$ exists since $S_m$ is a trace-class operator, and 
\begin{equation}
\lim_{n\to\oo}\langle \si_e\si_f \rangle^2=\det(I-S_m)\big|_{\epsilon=0}, \qq m \ge 0,\label{nii}
\end{equation}
as above.
We claim that there exists $C>0$, $\alpha\in(0,1)$, independent of $m,n$, such that
\begin{equation}
\left|\langle\si_e\si_f\rangle^2-\lim_{n\to\oo}\langle\si_e\si_f\rangle^2\right|\leq C\alpha^n,
\qq m,n \ge 0.
\label{edd}
\end{equation}
To show \eqref{edd}, first, following the same computations as above, we obtain
\begin{equation*}
\|S_{m,n}-S_m\|_1\leq C_7\beta_9^n,
\end{equation*}
for constants $C_7>0$, $\beta_9\in(0,1)$ independent of $m,n,\epsilon$. 

We consider the following cases:
\begin{enumerate}
\item[I.] $I-S$ is invertible,
\item[II.] $I-S$ is not invertible.
\end{enumerate}
\emph{In Case I}, for sufficiently large $m$, $I-S_m$ is also invertible. 
Following the  arguments of \cite[pp.\ 115--116]{ggk}, we obtain
\begin{equation*}
\bigl|\det(I-S_{m,n})-\det(I-S_m)\bigr|\leq C_8\beta_{10}^n,
\end{equation*}
for $C_8>0$, $\beta_{10}\in(0,1)$ independent of $m,n,\epsilon$.

\emph{In Case II}, the point $1$ is an isolated eigenvalue of finite type for $S$. 
Let $P$ be the corresponding Riesz projection, and put $H_1=\mathrm{Im}\, P$, 
$H_2=\mathrm{Ker}\, P$, so that $H_1$ is finite dimensional. With respect to the decomposition,
we have
\begin{equation*}
S_{m,n}=\begin{pmatrix}
K_{11}^{(m,n)}&K_{12}^{(m,n)}\\K_{21}^{(m,n)}&K_{22}^{(m,n)}\end{pmatrix},
\qq 
S_{m}=\begin{pmatrix}
K_{11}^{(m)}&K_{12}^{(m)}\\K_{21}^{(m)}&K_{22}^{(m)}\end{pmatrix}.
\end{equation*}
We now follow the argument of \cite[pp.\ 115--116]{ggk} and the proof in Section \ref{sec:pf32}, 
to obtain \eqref{edd}.

Fix $m$, and note that $S_m\bigr|_{\eps=0}$ is an operator of trace class. 
By \cite[Thm 4.6]{ZL1}
(see  also \cite{GK69}) and \eqref{nii},
$\lim_{n\to\oo}\langle\si_e \si_f\rangle^2=\det(I-S_m)\bigr|_{\eps=0}$ is analytic in $a,b,c$ subject to
\begin{equation}\label{eq:subsq}
a,b,c>0,\q \sqrt{a}<\sqrt{b}+\sqrt{c},\q  \sqrt{b}<\sqrt{a}+\sqrt{c},\q\sqrt{c}<\sqrt{a}+\sqrt{b}.
\end{equation}
By Remark \ref{rem:subcrit}, $\lim_{n\to\oo}\langle\si_e \si_f\rangle^2=0$
when
\begin{equation}\label{eq:subsq2}
a,b,c>0,\q a^2<b^2+c^2,\q b^2<a^2+c^2, \q c^2<a^2+b^2.
\end{equation}
The set of $(a,b,c)$ satisfying \eqref{eq:subsq2} is a subset of
the (connected) subset of $\RR^3$ given by \eqref{eq:subsq}, and it follows by analyticity that
\begin{equation}
\lim_{n\to\oo}\langle\si_e\si_f\rangle^2=0,\qq m \ge 0,
\end{equation}
subject to \eqref{eq:subsq}.

By the arguments that led to \eqref{edd},  
there exists $C>0$, $\alpha\in(0,1)$, independent of $m,n$, such that,
subject to \eqref{eq:subsq},
\begin{equation}
 \left|\langle\si_e\si_f \rangle^2-\lim_{m\to\oo}\langle\si_e\si_f\rangle^2\right|
\leq C\alpha^m,\qq m,n \ge 0,\label{edd1}
\end{equation}
and, in addition, 
\begin{equation}
\lim_{m\to\oo}\langle\si_e\si_f\rangle^2=0,\qq n \ge 0,\\
\label{nn}
\end{equation}
We combine \eqref{edd}--\eqref{nn} to obtain
\begin{equation*}
\langle\si_e\si_f\rangle^2\leq C\alpha^{m\vee n},
\end{equation*}
which implies \eqref{eds} with amended $C$, $\al$.

\end{proof}

\section{Periodic 1-2 models}\label{sec:periodic}

\subsection{Two-edge correlation for periodic models}
Some of the previous results may be extended
to certain periodic models, as explained next. Since each edge of $\HH$ touches exactly one white vertex,
a 1-2 model may be specified by assigning a parameter-vector $(a_w,b_w,c_w)$ to
each \emph{white} vertex  $w$.  For $k,l\in\NN$,
we call the ensuing model $ k\times l$ \emph{periodic} if 
$$
(a_w,b_w,c_w) = (a_v,b_v,c_v), \qq v=\tau_1^k\tau_2^l w,
$$
where the maps $\tau_i$ are illustrated in Figure \ref{fig:hex0}. 

By the techniques of Sections \ref{sec:morepf}, 
if the parameter-vectors of a periodic model are such that the associated spectral curve 
does not intersect the unit torus, then  the entries of the corresponding inverse Kasteleyn matrix 
$K^{-1}(v,w)$ converge to 0 exponentially as $|v-w|\to\oo$. Following the procedure 
of Section \ref{sec:eecad}, we obtain the following (in which the notation of Theorem \ref{thm:main4} has been
adopted).

\begin{theorem}\label{thm:main5}
Let $k,l \ge 1$. Assume the  1-2 model is 
$k\times l$ periodic, and the spectral curve does not intersect the unit torus.
\begin{letlist}
\item The limit $\La:= \lim_{m,n\to\oo}\langle\si_e\si_f\rangle^2$ exists, and
\begin{equation}
\left|\langle\si_e\si_f \rangle^2-\La\right|
\leq C\alpha^{m \wedge n},\label{dr}
\end{equation}
where $C>0$, $\alpha\in(0,1)$ are constants independent of $m$, $n$.
If $m=0$, we have 
\begin{equation}
\left|\langle\si_e\si_f \rangle^2-\La\right|
\leq C\alpha^{n}.\label{dr}
\end{equation} 
\item If, in addition,  $\lim_{|e-f|\rightarrow\infty}\langle\sigma_e\sigma_f\rangle^2=0$,
 then $C>0$, $\al\in(0,1)$ may be chosen such that
\begin{equation}
|\langle \si_e\si_f \rangle|\leq C\alpha^{|e-f|}.\label{eds}
\end{equation}
\end{letlist}
\end{theorem}

\subsection{1-2 and Ising models with period $k \times 1$}\label{sec:lis2}

We discuss a special case of periodic 1-2 models,
namely, models with period $k\times 1$. 
By the forthcoming Theorem \ref{thm:periodic}, the spectral curves of such models
can intersect the unit torus only at real points. 
The conclusions of Theorem \ref{thm:main5} follow whenever the corresponding characteristic polynomial 
$P(z,w)$ satisfies $P(\pm1,\pm1)\neq 0$.  Similar arguments are valid for periodic Ising models,
as exemplified in the forthcoming Example \ref{ex:perI}, which illuminates the differences between
the assumptions of the current paper and those of \cite{Lis}.

Let $k \ge 2$ and consider a $k\times 1$ periodic 1-2 model.
Let $\HH_{\Delta}$ be as in Section \ref{sec:dimer}, and let $\HH_{\Delta,k,1}$ be the quotient graph of $\HH_{\Delta}$ under the weight preserving action of $\ZZ^2$. 
Note that $\HH_{\Delta,k,1}$ is a finite graph that can be embedded in a torus.
We can divide $\HH_{\Delta,k,1}$ into $k$ parts, each of which is bounded by a 
quadrilateral region enclosing a NE/SW edge of the original lattice
$\HH$, see Figure \ref{fig:1kd}. 

\begin{figure}[htbp]
\centerline{\includegraphics*[width=0.6\hsize]{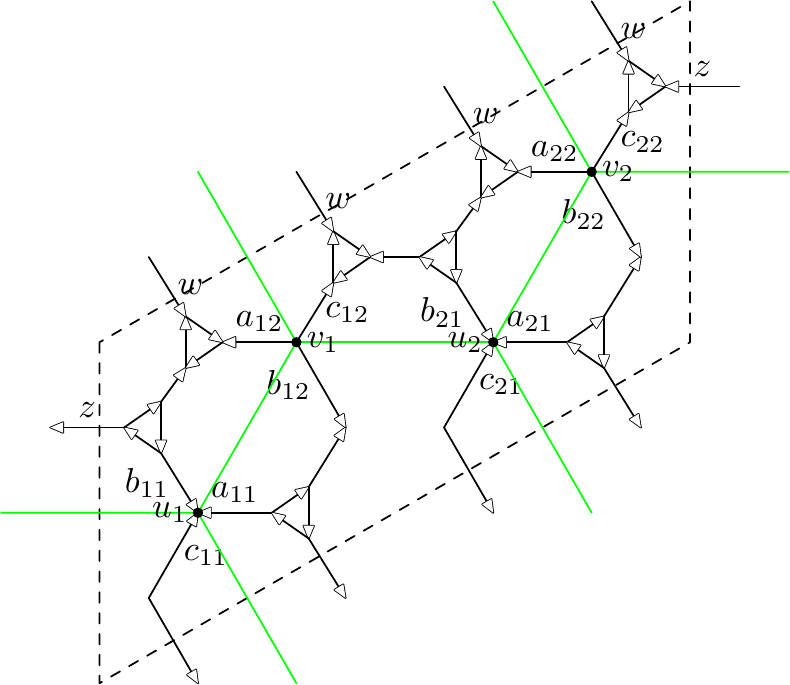}}
   \caption{The decorated graph $\HH_{\Delta,2,1}$; the hexagonal lattice $\HH$ is represented by green lines.}
   \label{fig:1kd}
\end{figure}

There are two vertices of $\HH$ lying in the $i$th such quadrilateral region
(see Figure \ref{fig:12con}), and to these
we assign local weights $(a_{i1},b_{i1},c_{i1})$ and $(a_{i2},b_{i2},c_{i2})$, respectively. 
As in Section \ref{ssec:spc},
we derive the modified weighted adjacency matrix $K(z,w)$ of $\HH_{\De,k,1}$, with 
characteristic polynomial $P(z,w)$. 

\begin{theorem}\label{thm:periodic}
The only possible intersection of the spectral curve $\{(z,w):P(z,w)=0\}$ 
with the unit torus 
$\TT^2=\{(z,w):|z|=1,\,|w|=1\}$ is a single real point.
\end{theorem} 

\begin{proof}
Note that a double dimer configuration is a union of cycles and 
doubled edges with the property that each vertex is incident to exactly two present 
edges. An (unoriented) $z$-\emph{edge} is an edge which has weight $z$ or 
$z^{-1}$ when oriented. 
Each double-dimer configuration of the determinant $P(z,w)$ falls 
into exactly one of the following two cases:
\begin{numlist}
\item it occupies each $z$-edge exactly once, i.e., each $z$-edge is in a cycle of the double dimer configuration.
\item each $z$-edge is either unoccupied or occupied exactly twice, i.e., 
each $z$-edge is either absent or a doubled edge in the double dimer configuration.
\end{numlist}

Each $k\times 1$ fundamental domain is comprised of $k$ $1\times 1$ blocks. For $1\leq i\leq n$, let $u_i,v_i$ be the two vertices of the hexagonal lattice lying in the $i$th block, see Figure \ref{fig:1kd}.

Each configuration in Case 1 consists of a single essential cycle of even length, together with some 
doubled edges.  The partition function of configurations in Case 1 is 
\begin{equation*}
P_1=\frac{1}{z}\prod_{i=1}^{k}\left[A_iw+B_i\right]+z\prod_{i=1}^{k}\left[\frac{A_i}{w}+B_i\right],
\end{equation*}
where
\begin{align*}
A_i=&-a_{i1}^2a_{i2}b_{i2} - a_{i1}a_{i2}^2b_{i1} - a_{i1}b_{i1}b_{i2}^2 + a_{i1}b_{i1}c_{i2}^2 - a_{i2}b_{i1}^2b_{i2} + a_{i2}b_{i2}c_{i1}^2,\\
B_i=&-a_{i1}^2a_{i2}c_{i2} - a_{i1}a_{i2}^2c_{i1} + a_{i1}b_{i2}^2c_{i1} - a_{i1}c_{i1}c_{i2}^2 + a_{i2}b_{i1}^2c_{i2} - a_{i2}c_{i1}^2c_{i2}.
 \end{align*}

Configurations in Case 2 depend on configurations of each $1\times 1$ block, which are 
determined by configurations of boundary edges. Let $Q_i^{00}$ (\resp, $Q_i^{22}$) 
denote the partition function at the $i$th block when both its $z$-edges are unoccupied 
(\resp, occupied). Let $Q_i^{20}$ (\resp, $Q_i^{02}$) denote the partition function at the $i$th 
block when its left (\resp, right) $z$-edge is occupied twice, while the right 
(\resp, left) edge is unoccupied. Then we have
\begin{align*}
Q_i^{00}&=\left(a_{i1}a_{i2} - \frac{b_{i1}c_{i2}}{w} + b_{i1}b_{i2} + c_{i1}c_{i2} - b_{i2}c_{i1}w\right)\\
&\hskip3cm \times \left(a_{i1}a_{i2} -\frac{ b_{i2}c_{i1} }{w}+ b_{i1}b_{i2} + c_{i1}c_{i2} - b_{i1}c_{i2}w\right),\\
Q_i^{02}&=W\left[a_{i2}^2b_{i1}c_{i1} + a_{i1}c_{i2}a_{i2}b_{i1} + a_{i1}b_{i2}a_{i2}c_{i1} + b_{i2}c_{i2}\left(b_{i1}^2 - b_{i1}c_{i1}\left(w+\frac{1}{w}\right) + c_{i1}^2\right)\right],\\
Q_i^{20}&=W\left[a_{i1}^2b_{i2}c_{i2} + a_{i2}c_{i1}a_{i1}b_{i2} + a_{i2}b_{i1}a_{i1}c_{i2} + b_{i1}c_{i1}\left(b_{i2}^2 - b_{i2}c_{i2}\left(w+\frac{1}{w}\right) +c_{i2}^2\right)\right],\\
Q_i^{22}&=-b_{i1}b_{i2}c_{i1}c_{i2}\left(w-\frac{1}{w}\right)^2+(a_{i1}b_{i2}+a_{i2}b_{i1})^2+(a_{i1}c_{i2}+a_{i2}c_{i1})^2\\
&\hskip3cm +\left(a_{i2} ^2 b_{i1} c_{i1} + a_{i1}^2  b_{i2} c_{i2} + a_{i1} a_{i2} b_{i1} c_{i2} + a_{i1} a_{i2} b_{i2} c_{i1}\right)\left(w+\frac{1}{w}\right),
 \end{align*}
where $W=(w-w^{-1})$.

Let $z_i$ be the $z$-edge connecting the $i$th $1\times 1$ 
block 
and the $[(i+1)\mod k]$th $1\times 1$ block, and let 
$t_i\in\{0,1,2\}$ denote the occupation time of $z_i$ by a given double dimer
configuration, which is to say that
$$
t_i =\begin{cases} 0 &\text{if $z_i$ is absent},\\
2 &\text{if $z_i$ is a doubled edge},\\ 
1 &\text{if $z_i$ is in a cycle}.
\end{cases}
$$
The partition function in Case 2 is 
\begin{equation*}
P_2=\sum_{t_1,\dots,t_k\in\{0,2\}}\prod_{i=1}^{i}Q_i^{t_{i-1}t_{i}}.
\end{equation*}
Let $w=e^{i\phi}$, so that $Q_i^{00},Q_i^{22}\geq 0$. 
By periodicity, in each configuration, the number of $Q^{02}$ blocks equals 
the number of $Q^{20}$ blocks. Therefore in all terms in $P_{2}$ where $\sin\phi$ 
appears, which are exactly the terms where $Q^{02}$ and $Q^{20}$ appear, $\sin\phi$ 
has even degree. Moreover, given that all the local weights are strictly positive,  
each term in the expansion of $P_2$ with $\sin\phi$ is non-negative. Therefore,
\begin{align*}
P(z,w)=P_1+P_2
=P_1+\prod_{i=1}^{k}Q_i^{00}+\prod_{i=1}^{k}Q_i^{22}+F(w),
\end{align*}
where $F(w)\geq 0$ is the sum of all terms in $P_2$ in which $Q^{02}$ and $Q^{20}$ appear at some 
$1\times 1$ blocks. Let $G(z,w)=z[P(z,w)-F(w)]$. For given $w$, $G(z,w)$ is a quadratic 
polynomial in $z$. 
Let
\begin{align*}
C_0=\prod_{i=1}^{k}\left[A_iw+B_i\right],\qq
C_1=\prod_{i=1}^{k}Q_i^{00}+\prod_{i=1}^{k}Q_i^{22}.
\end{align*}
The roots of $G(\cdot, w)$ are
\begin{equation}
z^\pm=\frac{C_1\pm\sqrt{C_1^2-4|C_0|^2}}{2\ol{C_0}}.\label{z12}
\end{equation}
Let $w=e^{i\phi}$, and note that
\begin{equation*}
C_1^2-4|C_0|^2 \geq 4\left(\prod_{i=1}^{k}[Q_i^{00}Q_i^{22}]-\prod_{i=1}^{k}\left[A_i^2+B_i^2+2A_iB_i\cos\phi\right]\right).
\end{equation*}
Moreover, 
\begin{align*}
&Q_i^{00}Q_i^{22}-\left(A_i^2+B_i^2+2A_iB_i\cos\phi\right)\\
&\q=-\left(w-\frac{1}{w}\right)^2\\
&\qq\times\left(a_{i1}^2b_{i2} c_{i2}  + a_{i2} c_{i1} a_{i1} b_{i2}  + a_{i2} b_{i1} a_{i1} c_{i2}+b_{i1} c_{i1}\left( b_{i2}^2+c_{i2}^2-2b_{i2}c_{i2}\cos\phi\right) \right)\\
&\qq\times\left(a_{i2}^2  b_{i1} c_{i1} + a_{i1} c_{i2} a_{i2} b_{i1} + a_{i1} b_{i2} a_{i2} c_{i1} + b_{i2} c_{i2} (b_{i1}^2  -  2b_{i1} c_{i1} \cos\phi  -c_{i1}^2 )\right).
\end{align*}

Let $a_i,b_i,c_i> 0$. Then $C_1^2-4|C_0|^2\geq 0$ with equality only if $w$ is real. 
If $C_1^2-4|C_0|^2>0$, then $|z^\pm|\neq 1$ and $G(\cdot,w)$ has no zeros on the unit circle. 
Hence $G(z,w)$ has no zeros on $\TT\times (\TT\setminus\{\pm 1\})$.
Since there exists at least one pair $\theta,\tau\in\{0,1\}$ with 
$$
P\bigl((-1)^{\theta},(-1)^{\tau}\bigr)=\det K\bigl((-1)^{\theta},(-1)^{\tau}\bigr)>0,
$$
we have that $G(z,w)\geq 0$ on $\TT^2$, with equality  
only if $w$ is real. By \eqref{z12}, when $w$ is real and $C_1^2-4|C_0|^2=0$, then $z^+=z^-$ is real. 
Therefore, the only possible intersection 
of $P(z,w)$ with $\TT^2$ is a single real point.
\end{proof}

The above technique applies also to the spectral curve of the
Ising model (not necessarily ferromagnetic) with $k\times 1$ periodicity on the triangular lattice.
The details are omitted.

\begin{proposition}\label{mim}
Consider a periodic Ising model on the triangular lattice. To each edge $e$, 
associate a coupling constant $J_e\in \RR$. Assume the coupling constants are 
translation-invariant with period $k\times 1$ where $k\ge 1$. 
The only possible intersections of the spectral curve with the unit torus are real points.
\end{proposition}

\begin{example}\label{ex:perI}
Here is an example which explores the generality of the arguments of the current paper.
We start with an Ising model on the triangular lattice $\TT$, with edge interactions $J_e\in \RR$
that are periodic with period $2\times 1$. Let $\HH$ be the dual hexagonal lattice of $\TT$.
The corresponding Fisher graph $\FF$
is obtained from $\HH$ by replacing each vertex by a triangle. 
Each triangle edge is assigned weight $1$, and a non-triangle edge crossing an edge $e$ of $\TT$
has weight $e^{2J_e}$. 

The dimer model on $\FF$ with the above edge-weights corresponds to the Ising model on the triangular lattice.  Note that the spins of the Ising model are placed at centres
of the dodecagons of $\FF$. Two adjacent spins have the same state
(\resp, opposite states) if
and only if the corresponding non-triangle edge of $\FF$ separating the two dodecagons
are present (\resp, absent). 

See Figure \ref{fig:12fisher} for an illustration, where $a_1,b_1,c_1,a_2,b_2,c_2$ are 
the edge-weights $e^{2J_e}$.

\begin{figure}[htbp]
\centerline{\includegraphics*[width=0.3\hsize]{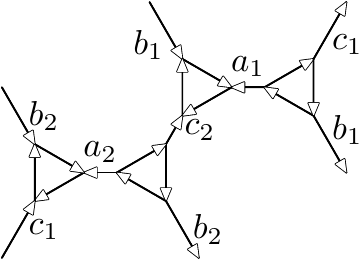}}
   \caption{The Fisher graph with $2\times 1$ periodic edge weights}
   \label{fig:12fisher}
\end{figure}

Given the clockwise-odd orientation of Figure \ref{fig:12fisher}, we can compute the 
characteristic polynomial $P(z,w)=\det K(z,w)$, where $K(z,w)$ is the modified 
weighted adjacency matrix whose rows and columns are indexed by 
vertices in the $2\times 1$ fundamental domain. 
By Corollary \ref{mim}, when $a_1,b_1,c_1,a_2,b_2,c_2>0$, the only possible intersection of $P(z,w)$ 
with the unit torus are real.

The function $P(z,w)$ may be calculated as in Section \ref{ssec:spc}, and it may  be checked that,
when $(a_1,b_1,c_1,a_2,b_2,c_2)=(1.1,0.9,0,0,0.5,0.5)$,  
\begin{equation}\label{eq:suff}
P(1,1)P(1,-1)P(-1,1)P(-1,-1)\neq 0.
\end{equation}
Let $(a_1,b_1,b_2,c_2)=(1.1,0.9,0.5,0.5)$ and assume that $c_1$, $a_2$ are positive and sufficiently small
that \eqref{eq:suff} continues to hold. By Proposition \ref{mim}, 
the spectral curve does not intersect the unit torus.
As in \cite{BdeT10,ckp00,KOS06}, 
the Ising free energy may be expressed in the form \eqref{fef}.
However,  when $(a_1,b_1,b_2,c_2)=(1.1,0.9,0.5,0.5)$ and $c_1$, $a_2$ are positive and sufficiently small, 
then neither the high-temperature  
nor the low-temperature condition of \cite{Lis} is satisfied.

Moreover, using the technique of \cite{ZL1}, the square of the spin--spin
 correlation may be expressed as the determinant of a block Toeplitz matrix. 
 By applying similar techniques as in 
 Sections \ref{sec:pf32}--\ref{sec:eecad}, we can obtain the convergence rate of the 
 spin--spin correlation of the Ising model of \eqref{thm:main5} whenever the spectral curve does 
 not intersect the unit torus.
\end{example}

\subsection{Harnack curve}

We present next another sufficient condition for the spectral curve 
$P(z,w)=0$ to intersect the unit torus at only real points. 
The exponential convergence 
rate for two-edge correlation functions follows by Theorem \ref{thm:main5}. 

Harnack curves were studied in \cite{GM00,MR01}. Simply speaking, a Harnack curve is the
 real part of a real algebraic curve $A$ (real zeros of a Laurent polynomial with real coefficients)  
 such that the map \eqref{logr} from $A$ to $\RR^2$ is at most two-to-one. It was proved 
 in \cite{KO06,KOS06} that the 
 spectral curve of any positive-weight, bi-periodic, planar, bipartite dimer model is a Harnack curve. 
 Using the combinatorial results of \cite{Dub11}, we infer that 
 the spectral curve of any ferromagnetic, bi-periodic, planar Ising model is also a Harnack curve. 
 If the Ising model is not ferromagnetic, there may exist a concrete counterexample in 
 which the spectral curve is not Harnack, \cite{RKp}. We give here a simple proof that, under 
 certain conditions, the assumption that the spectral curve is Harnack implies that its intersections 
 with the unit torus are necessarily real.

\begin{proposition}
Let $P(z,w)$ be a Laurent polynomial taking real values on the unit torus. 
If $A:=\{(z,w)\in\CC^2: P(z,w)=0\}$ is a Harnack curve, then $A$ can only intersect the unit torus 
$\TT^2$ at real points.
\end{proposition}

\begin{proof}
Define the logarithmic Gaussian map $\gamma_P:A\rightarrow \CC P^1$ by
\begin{equation*}
\gamma_P(z,w)=\left(z\frac{\partial P}{\partial z},w\frac{\partial P}{\partial w}\right),
\end{equation*}
and also $\Log: A\rightarrow \RR^2$ by
\begin{equation}
\Log: (z,w)\mapsto (\log|z|,\log|w|)\label{logr}.
\end{equation}
By \cite[Lemma 5]{GM00}, $A$ is Harnack, whence the real zeros satisfy
$A\cap \RR^2=\gamma_{P}^{-1}(\RR P^1)$. By \cite[Lemma 3]{GM00}, 
$\gamma_P^{-1}(\RR P^1)$ consists of the singular points of the map $\Log$. 

Let $(z,w)=(e^{i\theta},e^{i\phi})\in A\cap\TT^2$. 
Since $P(z,w)$ is a Laurent polynomial, we have that 
$$
\frac{\partial P}{\partial \ol{z}}=0,\qq \frac{\partial P}{\partial \ol{w}}=0,
$$
and hence
\begin{equation*}
\gamma_P(z,w)=\left(i\frac{\partial P}{\partial \theta},i\frac{\partial P}{\partial \phi}\right).
\end{equation*}
Given that $P$ takes real values on $\TT^2$, $\gamma_P(z,w)\in \RR P^1$. Hence $A\cap \TT^2\subseteq \gamma_P^{-1}(\RR P^1)=A\cap \RR^2$. Therefore, any zero of $P(z,w)$ on $\TT^2$ is real. 
\end{proof}

\section*{Acknowledgements} 
This work was supported in part
by the Engineering and Physical Sciences Research Council under grant EP/I03372X/1. 
ZL's research was supported by the Simons Foundation grant $\#$351813 and 
National Science Foundation DMS-1608896. We thank the referee for
a detailed and useful report.

\providecommand{\bysame}{\leavevmode\hbox to3em{\hrulefill}\thinspace}
\providecommand{\MR}{\relax\ifhmode\unskip\space\fi MR }
\providecommand{\MRhref}[2]{%
  \href{http://www.ams.org/mathscinet-getitem?mr=#1}{#2}
}
\providecommand{\href}[2]{#2}


\end{document}